\documentclass[11pt,reqno]{amsart}
\usepackage{amsmath,amssymb,mathrsfs,amsthm}
\usepackage[title]{appendix}
\usepackage{enumitem}
\usepackage{xcolor}  
\usepackage[colorlinks=true,linkcolor=blue,citecolor=red,urlcolor=cyan]{hyperref}

\let\cal=\mathcal
\def\N{{\mathbb N}}
\def\Z{{\mathbb Z}}
\def\R{{\mathbb R}}
\def\P{{\mathbb P}}
\def\E{{\mathbb E}}
\def\T{{\mathbb T}}

\def\cA{{\mathcal A}}

\def\var{\varepsilon}

\newtheorem{thm}{Theorem}[section]

\newtheorem{lem}[thm]{Lemma}
\newtheorem{prop}[thm]{Proposition}

\theoremstyle{definition}
\newtheorem{de}[thm]{Definition}
\theoremstyle{remark}
\newtheorem{rem}[thm]{Remark}
\newtheorem{exam}[thm]{Example}
\numberwithin{equation}{section}

\newcommand{\Law}[1]{\mathcal{L}_{#1}}
\newcommand{\norm}[1]{\left\| #1 \right\|}
\newcommand{\inpro}[1]{\left\langle #1 \right\rangle}

\newcommand{\wts}[1]{\left(\E\norm{#1}^2\right)^\frac{1}{2}}
\allowdisplaybreaks

\begin{document}
\title[Averaging Principle and Pullback Attractor Convergence]{Averaging Principle and Pullback Attractor Convergence for McKean--Vlasov Stochastic Reaction--Diffusion Equations} 

\author{Honglei Chen}
\address{H. Chen: School of Mathematical Sciences, Dalian University of Technology, Dalian
116024, P. R. China}
\email{hcl9712285@gmail.com}

\author{Mengyu Cheng}
\address{M. Cheng: School of Mathematical Sciences, Dalian University of Technology, Dalian
116024, P. R. China}
\email{mycheng@dlut.edu.cn}

\author{Zhenxin Liu}
\address{Z. Liu: School of Mathematical Sciences, Dalian University of Technology, Dalian
116024, P. R. China}
\email{zxliu@dlut.edu.cn}
   
\date{August 4, 2026}  



\begin{abstract}
We establish three averaging principles for distribution-dependent stochastic
reaction--diffusion equations with rapidly oscillating coefficients on the
torus $\T^d$, $d\le3$. First, solutions converge in mean square, uniformly on
finite time intervals, to solutions of the averaged equation. Under a
contraction condition, both the original and averaged equations admit unique bounded entire solutions
whose mean-square distance vanishes uniformly for all
$t\in\mathbb R$. At the level of
probability laws, the original nonautonomous equation possesses a family of
pullback attractors, whereas the averaged equation has a
global attractor; the former converge upper-semicontinuously to the latter,
uniformly over the coefficient hull.  As an application, 
we present a class of stochastic reaction--diffusion models 
motivated by large-scale interacting systems.
\end{abstract}

\maketitle

\section{Introduction}
In this paper, we establish three types of averaging principle--the first and
second Bogolyubov theorems and the global averaging principle--for a class of
McKean--Vlasov stochastic reaction--diffusion equations with rapidly
oscillating coefficients on the torus $\T^d$, $d\le3$.
For \(0<\var\ll1\), the equation under consideration is
\begin{equation}\label{eq:SPDEone}
\begin{aligned}
du_\var(t)&=\left[
\Delta u_\var(t)
- \lambda u_\var(t) 
- \gamma(\frac{t}{\var})|u_\var(t)|^2 u_\var(t) 
+ f\left(\frac{t}{\var}, u_\var(t), \Law{u_\var(t)}\right)\right]dt \\
&\qquad\qquad+ g\left(\frac{t}{\var}, u_\var(t), \Law{u_\var(t)}\right)\,dW_t,
\end{aligned}
\end{equation}
where $\mathcal L_{u_\varepsilon(t)}$ denotes the law of
$u_\varepsilon(t)$,
$W_t$ is a cylindrical Wiener process on a separable Hilbert space $U$,
$f:\R\times H\times \mathcal{P}_2(H)\rightarrow H$
and $g:\R\times H\times \mathcal{P}_2(H)\rightarrow L_2(U,H)$.
Here $H$ is the zero-mean subspace of $L^2(\T^d)$ introduced in \eqref{eq0805:02},
and $\mathcal P_2(H)$ denotes the space of Borel 
probability measures on $H$ with finite second moments.
The dependence of \(f\) and \(g\) on \(\Law{u_\var(t)}\) represents a
mean-field interaction: the drift and noise at time \(t\) may depend not
only on the current state but also on the statistical distribution of the
population. The illustrative model in Section~2.3, motivated by continuum
neural-field models, makes this mechanism explicit: its distribution
dependence enters through the population-mean feedback
\(m(t)\int_H z\,\Law{u_\var(t)}(dz)=m(t)\E u_\var(t)\). The modelling
background is provided by Bressloff's review of continuum neural fields
\cite{bressloff2012spatiotemporal} and the monograph of Ermentrout and
Terman \cite{ermentrout2010mathematical}.

Averaging principles are among the basic tools for the analysis of dynamical systems with separated time scales. In their classical form, they show that the slow component of a rapidly oscillating system can be approximated, on finite time intervals and sometimes on longer time scales, by an effective equation in which the fast oscillations have been averaged out. The first rigorous results go back to the work of Krylov, Bogolyubov and Mitropolsky \cite{bogolyubov1961asymptotic,krylov1943introduction}, and were later extended to stochastic differential equations by Khasminskii \cite{has1966stochastic,khasminskij1968principle,khasminskii2004averaging}. Since then, finite-dimensional stochastic averaging has been developed for fully coupled systems, time-dependent coefficients, fractional noise, and fluctuation and large-deviation problems; see, for instance,
\cite{bakhtin2004diffusion,freidlinWentzell1998,kifer2004some,
veretennikov1991averaging,veretennikov1999large,liu2020averaging,
hairer2020averaging,pei2021averaging,
cheng2023averaging2}.
Long-time and uniform-in-time averaging, including periodic or almost-periodic
settings, has been studied in
\cite{freidlin2006long,kamenskii2015weak,cheban2021averaging,
liuwang2021averaging,uda2021averaging,wainrib2013double,
crisan2026uniform, xie2026uniform}.

The averaging principle has also been extensively developed for stochastic
partial differential equations under a wide range of structural assumptions
and noise regimes; see, e.g.,
\cite{cerrai2009khasminskii,cerrai2009averaging,cerrai2011averaging,
brehier2012strong,wang2012average,duan2014effective,fu2015strong,
bao2017twotime, pei2017stochastic, cerrai2017averaging,
dong2018burgers, gao2019bogoliubov}
and the more recent works
\cite{sun2021averaging,hong2022mckean,
cheng2023second,cheng2023averaging,defeo2023order,
MR2023Asymptotic,
cheng2026nonautonomous}, together with the references therein.

McKean--Vlasov stochastic equations, also known as distribution-dependent stochastic
equations, are fundamental to the probabilistic description of collective
phenomena. Initiated by McKean's nonlinear Markov-process formulation
\cite{mckean1966class}, they arise as mean-field limits of large systems of
weakly interacting particles, with propagation of chaos connecting the
particle dynamics to the limiting nonlinear equation; see the classical
account \cite{sznitman1991topics}. They also provide a basic framework for
nonlinear Fokker--Planck equations \cite{buckdahn2017meanfield}, mean-field
control, and mean-field games \cite{lasry2007mean,carmona2018probabilistic}.
These connections make the derivation of effective dynamics under separated
time scales or rapidly varying environments particularly relevant.

Against this background, averaging principles for McKean--Vlasov SDEs
have been established for slow-fast systems; see, for example,
\cite{rockner2021mckean,xu2021twotime,
shen2022averaging,cheng2024ddsde,LY2024Poisson,
shen2025twotime,shi2025averaging,yin2026stability}.
To the best of our knowledge, these works are confined to the first
Bogolyubov theorem, namely approximation on finite time intervals.

It is therefore natural to ask whether, in the distribution-dependent SPDE
setting, averaging persists on the whole real line and at the level of
asymptotic dynamics. With this motivation, we now turn to the main results of
this paper, which comprise three parts.

First, we prove a finite-time Bogolyubov averaging principle for the distribution-dependent stochastic reaction--diffusion equation \eqref{eq:SPDEone}. More precisely, if \(u_\var(t)\) denotes the solution of \eqref{eq:SPDEone} with the initial data $u_0$,
then for every \(T>0\),
\[\lim_{\var\to0}\E\left(\sup_{t\in[0,T]}\|u_\var(t)-\bar u(t)\|^2\right)=0,\]
where $\bar{u}(t),t\in[0,T]$ is the solution to
\begin{equation}\label{eq:SPDEtwo}
d\bar{u}(t)
= \left[
\Delta \bar{u}(t)
- \lambda \bar{u}(t) 
- \bar{\gamma}|\bar{u}(t)|^2 \bar{u}(t)
+ \bar{f}\left(\bar{u}(t), \Law{\bar{u}(t)}\right)
\right] dt 
+ \bar{g}\left(\bar{u}(t), \Law{\bar{u}(t)}\right)\,dW_t
\end{equation}
satisfying $\bar{u}(0)=u_0$. Here the averaged coefficients 
$\bar{\gamma}$, $\bar{f}$ and $\bar{g}$ are 
specified in \ref{item:H-a}.
Finite-time averaging has been established for McKean--Vlasov SDEs and for
slow--fast McKean--Vlasov SPDEs; see, for example,
\cite{hong2022mckean,cheng2024ddsde}. Those results, however, do not cover the
present single reaction--diffusion equation, in which the rapid time
oscillations enter the coefficient of the non-Lipschitz cubic term as well as
the law-dependent drift and diffusion. Thus, the convergence above provides a
finite-time strong averaging result for a class of distribution-dependent
reaction--diffusion equations beyond the existing frameworks.

Second, under an additional contraction condition, we extend the averaging principle from finite intervals to the whole real line. In this case,
\begin{equation}\label{eq0805:01}
\lim_{\var\to0}\sup_{t\in\R}\E\|u_\var(t)-\bar u(t)\|^2=0,
\end{equation}
where $u_\var(t),t\in\R$ and $\bar{u}(t),t\in\R$ are the unique bounded
entire solutions to \eqref{eq:SPDEone} and \eqref{eq:SPDEtwo} respectively.

Compared with \cite{cheng2023second,cheng2023averaging}, the present equation
allows the coefficients to depend on the law of the solution. Moreover, our
second result \eqref{eq0805:01} gives mean-square convergence of the bounded
entire solutions uniformly for all $t\in\mathbb R$, whereas
\cite{cheng2023second,cheng2023averaging} also consider bounded entire
solutions defined on $\mathbb R$ but establish only convergence in
distribution, uniformly on each bounded time interval. Cheban and Liu
\cite{cheban2021averaging} and Liu and Wang \cite{liuwang2021averaging} also
established whole-line averaging, but for 
distribution-independent stochastic ordinary differential equations and in
the sense of convergence in distribution. Thus, even at the whole-line level,
our result differs both in the class of equations and in the mode of
convergence: it gives mean-square convergence for the unique bounded entire
solutions of a distribution-dependent SPDE. 

More recently, Crisan et al.
\cite{crisan2026uniform} and Xie and Zhang \cite{xie2026uniform} obtained
quantitative uniform-in-time approximations for finite-dimensional multiscale
SDEs starting from prescribed initial data, with estimates uniform for
$t\geq0$. These are forward-time results, whereas \eqref{eq0805:01} is uniform
on the whole real line $t\in\mathbb R$.

Third, we study the long-time behavior of the probability laws generated by
\eqref{eq:SPDEone} and \eqref{eq:SPDEtwo}. Since the coefficients of
\eqref{eq:SPDEone} depend on time, its dynamics is nonautonomous and must be
considered over the hull $\mathcal H(F_0)$, namely the closure of the time
translates of $F_0=(\gamma,f,g)$. We prove that, for each
$F\in\mathcal H(F_0)$, the corresponding law dynamics admits a pullback
attractor $\mathcal A^\varepsilon(F)\subset\mathcal P_2(H)$, whereas the
autonomous averaged equation admits a global attractor
$\bar{\mathcal A}\subset\mathcal P_2(H)$. Moreover, these attractors converge
uniformly with respect to the coefficient symbol:
\[
\lim_{\varepsilon\to0}\sup_{F\in\mathcal H(F_0)}
\operatorname{dist}\bigl(\mathcal A^\varepsilon(F),\bar{\mathcal A}\bigr)=0.
\]

The main difficulty in proving attractor convergence is a mismatch in moment
regularity. The finite-time averaging theorem (Theorem \ref{thmT}) is initially available for
initial laws in $\mathcal P_4(V)$, whereas the attractors are constructed in
$\mathcal P_2(H)$ and their elements need not have finite fourth moments. We
overcome this difficulty by proving that $\mathcal P_4(V)$ is dense in
$\mathcal P_2(H)$ and using continuous dependence on the initial law to extend
the finite-time averaging result uniformly to compact subsets of
$\mathcal P_2(H)$. A further difficulty is pullback asymptotic compactness in
the Wasserstein metric: the tightness obtained from the pullback
$H^1$-bounds and the compact embedding $H^1(\mathbb T^d)\hookrightarrow H$
is not sufficient for compactness in $W_2$. We therefore establish a uniform
second-moment tail estimate, which provides the required uniform integrability
and, together with tightness, yields relative compactness in
$\mathcal P_2(H)$; see Lemma \ref{lem:uniformac} for details.

At the level of attractors, the earlier works \cite{cheng2023second, cheng2023averaging}  encode
the nonautonomous long-time dynamics through a uniform measure attractor,
whereas we retain the coefficient symbol and study a family of pullback
attractors over a hull that need not be compact. We then obtain an
upper-semicontinuity result that is uniform with respect to the symbol.
Thus the three-part structure is parallel to these earlier works, while the
law dependence and the noncompact base dynamics require different estimates
and compactness arguments.

The paper is organized as follows. In Section 2 we introduce the functional setting, state the assumptions, and formulate the main results. In Section 3 we prove the finite-time and whole-line averaging principles for \eqref{eq:SPDEone}. In Section 4 we study the induced dynamics on \(\cal P_2(H)\), prove the existence of pullback and global attractors, and establish the convergence of attractors. 

$\textbf{Notation.}$ Throughout this paper, we work on the torus $\T^d$ with $d \leq 3$. We denote
\begin{equation}\label{eq0805:02}
H := \left\{ u \in L^2(\T^d): \int_{\T^d} u(x)\,dx = 0 \right\}, 
\quad
V := \left\{ u \in H^1(\T^d): \int_{\T^d} u(x)\,dx = 0 \right\}.
\end{equation}
The norm and inner product are denoted by
\[\norm{u}:=\norm{u}_{L^2},\quad\norm{u}_V^2:=\|\nabla u\|_{L^2}^2+\lambda\|u\|_{L^2}^2,\]
\[\inpro{u,v}:=\int_{\T^d}u(x)v(x)\,dx,\]
and let $A:=\Delta -\lambda I$ on $H$.
Let \(P_0:L^2(\T^d)\to H\) be the orthogonal projection onto the zero-mean subspace. 
All nonlinear terms are understood after projection onto $H$; 
in particular, the cubic term means $P_0(|u|^2u)$. We suppress $P_0$ because
for $u\in H$, $\inpro{P_0(|u|^2u),u}=\inpro{|u|^2u,u}$.
On this zero-mean subspace, the Poincar\'e inequality holds with a constant $\lambda_1>0$:
\begin{align}\label{lambda1}
\norm{u}^2 \le \frac{1}{\lambda_1}\norm{\nabla u}^2, \quad u \in H.
\end{align}
We equip $\cal P_2(H)$ with the Wasserstein distance $W_2$, defined by
\[W_2^2(\mu,\nu):=\inf_{\pi\in\Pi(\mu,\nu)}\int_{H\times H}\|x-y\|^2\,\pi(dx,dy),\]
where $\Pi(\mu,\nu)$ is the set of probability measures on $H\times H$ with marginals $\mu$ and $\nu$. Equivalently,
\[W_2^2(\mu,\nu)=\inf \E\|X-Y\|^2,\]
where the infimum runs over all pairs $(X,Y)$, defined on a common probability space, with laws $\mu$ and $\nu$.
Denote by $\delta_0$ the Dirac measure at $0\in H$. Define
\[dist(B_1,B_2):=\sup_{a\in B_1}\inf_{b\in B_2}W_2(a,b),\qquad B_1,B_2\subset\cal P_2(H).\]
Let $\lfloor C\rfloor$ denote the integer part of $C$ for any $C\geq0$.
We use $C$ with or without subscripts to denote some constant, 
which may change from line to line.

\section{Statement of Main Results}

Let $(\Omega,\cal F,\{\cal F_t\}_{t\in\R},\P)$ be a filtered probability space satisfying the usual conditions, and let $(W_t)_{t\in\R}$ be a two-sided cylindrical Wiener process on a separable Hilbert space $U$, adapted to $\{\cal F_t\}_{t\in\R}$. Consider the stochastic reaction--diffusion equation
\begin{align*}
du_\var(t)&=\left[
\Delta u_\var(t)
- \lambda u_\var(t) 
- \gamma\left(\frac{t}{\var}\right)|u_\var(t)|^2 u_\var(t) 
+ f\left(\frac{t}{\var}, u_\var(t), \Law{u_\var(t)}\right)\right]dt \\
&\qquad\qquad+ g\left(\frac{t}{\var}, u_\var(t), \Law{u_\var(t)}\right)\,dW_t,
\end{align*}
where $\lambda>0$, $\gamma\in C_b(\R,\R_+)$, and $f:\R\times H\times\cal P_2(H)\to H$, $g:\R\times H\times\cal P_2(H)\to L_2(U,H)$ are measurable mappings.

\subsection{Assumptions}\label{sec:assump}

Assume that the following conditions hold.
\begin{enumerate}[label=\textup{(H\arabic*)},leftmargin=2.7em]
\item\label{item:H-f}
There exist constants $K_f>0$ and $L_f>0$ such that for all $t\in\R$,
$x,y\in H$ and $\mu_1,\mu_2\in \cal{P}_2(H)$,
\[
\norm{f(t,0,\delta_0)}\le K_f,\quad\norm{f(t,x,\mu_1)-f(t,y,\mu_2)}\le L_f(\norm{x-y}+W_2(\mu_1,\mu_2)).
\]
\item\label{item:H-g}
(i) There exist constants $K_g>0$ and $L_g\ge0$ such that for all $t\in\R$, $x,y\in H$, and $\mu_1,\mu_2\in\cal{P}_2(H)$,
\[\|g(t,0,\delta_0)\|_{L_2(U,H)}\le K_g,\]
\[\|g(t,x,\mu_1)-g(t,y,\mu_2)\|_{L_2(U,H)}\le L_g(\|x-y\|+W_2(\mu_1,\mu_2));\]
(ii) There exist constants $K_g'>0$ and $L_g'\ge0$ such that for all $t\in\R$, $x,y\in V$, and $\mu_1,\mu_2\in\cal{P}_2(V)$,
\[\|g(t,0,\delta_0)\|_{L_2(U,V)}\le K_{g}',\]
\[\|g(t,x,\mu_1)-g(t,y,\mu_2)\|_{L_2(U,V)}\le L_{g}'\bigl(\|x-y\|_V+W_{2,V}(\mu_1,\mu_2)\bigr).\]
Here $W_{2,V}$ is the Wasserstein distance induced by the norm of $V$.

\item\label{item:H-a}
There exists $\omega^f:(0,+\infty)\to\R_+$ satisfying $\omega^f(T)\to 0$ as $T\to \infty$ such that for all $x\in H$ and $\mu\in \cal{P}_2(H)$,
\begin{align}
\label{wf}
\sup_{a\in\R}\norm{\frac{1}{T} \int_{a}^{a+T}(f(s,x,\mu)-\bar{f}(x,\mu))\,ds}\le \omega^f(T)(1+\norm{x}+(\mu(\norm{\cdot}^2))^\frac12),
\end{align}
where $\mu(\norm{\cdot}^2):=\int_{H}\norm{x}^2\mu(dx)$. 

There exists $\omega^g:(0,+\infty)\to\R_+$ satisfying $\omega^g(T)\to 0$ as $T\to \infty$ such that 
\begin{align}
\label{wg}
\sup_{a\in\R}\frac{1}{T} \int_{a}^{a+T}\norm{g(s,x,\mu)-\bar{g}(x,\mu)}_{L_2(U,H)}^2\,ds\le \omega^g(T)(1+\norm{x}^2+\mu(\norm{\cdot}^2))
\end{align}
for all $x\in H$ and $\mu\in \cal{P}_2(H)$.

There exists $\omega^\gamma:(0,+\infty)\to\R_+$ satisfying $\omega^\gamma(T)\to 0$ as $T\to \infty$ such that 
\begin{align}
\label{wgamma}
\sup_{a\in\R}\left| {\frac{1}{T} \int_{a}^{a+T}(\gamma(s)-\bar{\gamma})\,ds} \right|\le \omega^\gamma(T).
\end{align}

\item\label{item:H-d}
There exist constants $c_0\ge 0$, $c_1\in\R$, $c_2\ge0$ such that for all $t\in\R$, $u\in H$, and $\mu\in\cal{P}_2(H)$,
\begin{align*}
2\langle f(t,u,\mu),u\rangle+\|g(t,u,\mu)\|_{L_2(U,H)}^2
\le c_0+c_1\norm{u}^2+c_2\,\mu(\|\cdot\|^2).
\end{align*} 
\end{enumerate}

\begin{rem}\label{rem0618}
(i) 
By \ref{item:H-f}, \ref{item:H-g} and Young's inequality,
we have for any $\iota\in(0,1)$,
\begin{align}\label{eq0731:02-1}
\norm{f\left(t,x,\mu\right)}
\leq L_f\|x\|+L_f\mu(\|\cdot\|^2)^{\frac12}+K_f,
\end{align}
\begin{align}\label{eq0731:02}
\norm{g\left(t,x,\mu\right)}_{L_2(U,H)}^2
&
\leq (L_g^2+\iota)\|x\|^2+C_\iota\left(\mu(\|\cdot\|^2)+1\right),
\end{align}
and
\begin{equation}\label{eq0803:05}
\norm{g\left(t,x,\mu\right)}_{L_2(U,H)}^2
\leq (2L_g^2+\iota)\|x\|^2+(2L_g^2+\iota)\mu(\|\cdot\|^2)+C_\iota.
\end{equation}

(ii) Under conditions \ref{item:H-f}, \ref{item:H-g} and \ref{item:H-a},
the averaged coefficients $\bar f$ and $\bar g$ satisfy \ref{item:H-f},
\ref{item:H-g}, \eqref{eq0731:02-1} and \eqref{eq0731:02}. 

(iii) By Lemma \ref{lem:mild-solution}, we obtain that
for any \(u_0,\bar u_0\in L^2(\Omega,\cal F_0,\P;H)\), 
equations \eqref{eq:SPDEone} and \eqref{eq:SPDEtwo} admit unique mild solutions 
\(u_\var(\cdot)\) and \(\bar u(\cdot)\), respectively. 
Moreover, for all $T>0$,
\[
u_\var(\cdot),\bar u(\cdot)\in L^2\bigl(\Omega;C([0,T];H)\bigr)\cap L^2\bigl(\Omega;L^2([0,T];V)\bigr).
\]
\end{rem}

\begin{rem}\label{rem0618-2}
If \ref{item:H-a} and \ref{item:H-d} hold, then $\bar f$ and $\bar g$ also satisfy the same dissipativity condition, namely,
\[
2\inpro{\bar f(u,\mu),u}
+
\norm{\bar g(u,\mu)}_{L_2(U,H)}^2
\le
c_0+c_1\norm{u}^2+c_2\mu(\norm{\cdot}^2),
\qquad \forall u\in H,\ \mu\in\cal P_2(H).
\]
Indeed, this follows by averaging the inequality in \ref{item:H-d} and
using Jensen's inequality for the term $\|\bar g(u,\mu)\|_{L_2(U,H)}^2$.
\end{rem}

\subsection{Main results}

Our first result is the finite-time Bogolyubov theorem.
\begin{thm}\label{thmT}
Assume that \ref{item:H-f}, \ref{item:H-g} and \ref{item:H-a} hold. Let $u_\var(t)$ and
$\bar{u}(t)$ be the mild solutions of \eqref{eq:SPDEone} and
\eqref{eq:SPDEtwo}, respectively, with the common initial datum
$u_0\in L^4(\Omega,\cal{F}_0,\P;V)$. Then for every $T>0$,
\[
\lim_{\var\to0}\E\left(\sup_{t\in[0,T]}\|u_\var(t)-\bar{u}(t)\|^2\right) =0.
\]
\end{thm}

The following result is the second Bogolyubov theorem, 
which concerns the convergence of the unique bounded entire solutions.
\begin{thm}
\label{thmR}
Assume that \ref{item:H-f}, \ref{item:H-g} and \ref{item:H-a} hold. 
Suppose further that $\lambda_1+\lambda>2L_f+2L_g^2$.
Then the following conclusions hold.
\begin{enumerate}
\item[(i)] Equation \eqref{eq:SPDEone} admits a unique bounded entire mild solution
\(u_\var(t)\), \(t\in\R\), which is given by
\begin{align*}
u_\var(t)
=&
\int_{-\infty}^{t}S(t-s)\left[-\gamma\left(\frac{s}{\var}\right)|u_\var (s)|^2u_\var(s)
+f\left(\frac{s}{\var}, u_\var(s), \Law{u_\var(s)}\right)\right]\,ds \\
&
+\int_{-\infty}^{t}{S(t-s)g\left(\frac{s}{\var}, u_\var(s), \Law{u_\var(s)}\right)\,dW_s},
\end{align*}
and satisfies
\[
\sup_{t\in\R}\E\|u_\var(t)\|^2<\infty.
\]
Here $\{S(t)\}_{t\geq0}$ is the semigroup generated by $A:=\Delta-\lambda I$ on $H$.
\item[(ii)] Equation \eqref{eq:SPDEtwo} admits a unique bounded entire mild solution
\(\bar u(t)\), \(t\in\R\), which is given by
\begin{align*}
\bar{u}(t)=
&
\int_{-\infty}^{t}S(t-s)\left[-\bar{\gamma}|\bar{u}(s)|^2\bar{u}(s)
+\bar{f}\left(\bar{u}(s), \Law{\bar{u}(s)}\right)\right]\,ds \\
&
+ \int_{-\infty}^{t}{S(t-s)\bar{g}\left(\bar{u}(s), \Law{\bar{u}(s)}\right)\,dW_s},
\end{align*}
and satisfies
\[
\sup_{t\in\R}\E\|\bar u(t)\|^2<\infty.
\]

\item[(iii)] The bounded entire solutions satisfy the uniform regularity estimate
\begin{equation}\label{entire-V4}
\sup_{0<\var\le1}\sup_{t\in\R}\E\|u_\var(t)\|_V^4
+\sup_{t\in\R}\E\|\bar u(t)\|_V^4<\infty.
\end{equation}
Moreover, 
\begin{equation}\label{eq0802}
\lim_{\var\to0}
\sup_{t\in\R}
\E\|u_\var(t)-\bar u(t)\|^2=0.  
\end{equation}
\end{enumerate}
\end{thm}

We next state the global averaging principle on the probability measure space.
More precisely, the original system induces a cocycle on $\cal{P}_2(H)$ and admits pullback attractors
(see Definitions~\ref{def0617:01} and~\ref{def0617:02} for details), 
while the averaged system generates an autonomous semigroup with a global attractor. 
The global averaging principle states that the pullback attractors of the original system converge to 
the global attractor of the averaged system as the time scale $\varepsilon$ goes to zero.

Write $F_0=(\gamma,f,g)$ and let $\sigma_\tau F_0:=F_0(\cdot+\tau,\cdot)$.
For $l\in\N$, set
\[
K_l:=\left\{\mu\in\cal P_2(H):\int_H\|z\|^2\,\mu(dz)\le l^2\right\}
\]
and
\begin{align*}
d_l(F_1,F_2):=\sup_{\substack{|s|\le l,\ \|x\|\le l\\ \mu\in K_l}}
\big(&|\gamma_1(s)-\gamma_2(s)|+\|f_1(s,x,\mu)-f_2(s,x,\mu)\|\\
&+\|g_1(s,x,\mu)-g_2(s,x,\mu)\|_{L_2(U,H)}\big).
\end{align*}
We equip the coefficient space with the metric
\[
d(F_1,F_2):=\sum_{l=1}^\infty 2^{-l}\frac{d_l(F_1,F_2)}{1+d_l(F_1,F_2)}
\]
and define the hull
\[
\cal H(F_0):=\overline{\{\sigma_\tau F_0:\tau\in\R\}}.
\]
For each fixed \(\tau\in\R\), translation is a continuous bijection of the coefficient
space endowed with \(d\). Consequently, it leaves \(\cal H(F_0)\)
invariant, and
\[
\sigma_0=\operatorname{Id},\qquad
\sigma_{\tau_1+\tau_2}=\sigma_{\tau_1}\sigma_{\tau_2},
\qquad \tau_1,\tau_2\in\R.
\]
Thus translation defines a group action on $\cal H(F_0)$. 
For any $\tilde{F}\in\cal H(F_0)$, define
\[\Phi^\var(t,\tilde{F},\mu_0):=\Law{u_\var(t;0,\mu_0,\tilde{F})},\]
where $u_\var(t;0,\mu_0,\tilde{F})$ denotes the solution to \eqref{eq:SPDEone} driven by $\tilde{F}=(\tilde{\gamma},\tilde{f},\tilde{g})$ with initial law $\mu_0\in \cal{P}_2(H)$.
Set
\[\bar{P}_t^*(\mu_0) := \Law{\bar{u}(t;\mu_0)}, \]
where $\bar{u}(t;\mu_0)$ denotes the solution to the averaged equation with initial law $\mu_0$.  

\begin{thm}\label{thm:attractor-convergence}
Assume that \ref{item:H-f}, \ref{item:H-g}, \ref{item:H-a} and \ref{item:H-d} hold with
$2(\lambda_1+\lambda)>\max\{c_1+c_2,c_1+2L_g^2\}$.
Then we have the following conclusions.
\begin{enumerate}
\item For each \(0<\var\le1\), the cocycle $\Phi^\var$ generated by the original system \eqref{eq:SPDEone}, over the scaled translation action
$\sigma_t^\var F=\sigma_{t/\var}F$ on the coefficient hull $\cal H(F_0)$, admits a pullback attractor
$\{\cal A^\var(F)\}_{F\in \cal H(F_0)}$
in \(\cal P_2(H)\), which attracts every bounded subset of \(\cal P_2(H)\).

\item The averaged semigroup \(\{\bar P_t^*\}_{t\ge0}\) 
generated by
\eqref{eq:SPDEtwo} admits a global attractor \(\bar{\cal A}\subset\cal P_2(H)\).

\item The pullback attractors of the original system converge to the global
attractor of the averaged system:
\[
\lim_{\var\to0}
\sup_{F\in\cal H(F_0)}
dist\bigl(\cal A^\var(F),\bar{\cal A}\bigr)=0.
\]
\end{enumerate}
\end{thm}

\subsection{An illustrative model}

The following example illustrates the assumptions with a spatial
mean-field feedback term. Such terms occur in continuum models of
large interacting populations; see, for example,
\cite{bressloff2012spatiotemporal,ermentrout2010mathematical}.

\begin{exam}
Consider the following stochastic reaction--diffusion equation with mean-field feedback
\begin{equation}\label{eq:app-model}
\begin{aligned}
du_\var(t)
=&\Big[\Delta u_\var(t)-\lambda u_\var(t)-\gamma\left(\frac{t}{\var}\right)|u_\var(t)|^2u_\var(t)+a\left(\frac{t}{\var}\right)Bu_\var(t) \\
&\qquad\qquad+b\left(\frac{t}{\var}\right)h
+m\left(\frac{t}{\var}\right)\mathbb{E}u_\var(t)\Big]dt
+\left(1+\frac{1}{1+|\tfrac{t}{\var}|^\iota}\right)Q\,dW_t ,
\end{aligned}
\end{equation}
where $B$ is bounded on both $H$ and $V$ and preserves the zero-mean subspace, 
$h\in V$, $Q\in L_2(U,V)$ and $\iota>0$. 
Here $\gamma\in C_b(\R,\R_+)$ and $a,b,m\in C_b(\R)$ 
are almost periodic in the sense of Definition~\ref{defAP}.
Hence, $\bar\gamma,\bar a,\bar b$ and $\bar m$ exist, and the following limits are uniform in $t\in\R$:
\[
\bar{\gamma}=\lim_{T\rightarrow\infty}\frac{1}{T}\int_{t}^{t+T}\gamma(s)\,ds,\quad
\bar{a}=\lim_{T\rightarrow\infty}\frac{1}{T}\int_{t}^{t+T}a(s)\,ds,
\]
\[
\bar{b}=\lim_{T\rightarrow\infty}\frac{1}{T}\int_{t}^{t+T}b(s)\,ds,\quad
\bar{m}=\lim_{T\rightarrow\infty}\frac{1}{T}\int_{t}^{t+T}m(s)\,ds.
\]
Define
\[
f(t,x,\mu):=a(t)Bx+b(t)h+m(t)\int_H z\mu(dz),
\qquad
g(t,x,\mu):=\left(1+\frac{1}{1+|t|^\iota}\right)Q.
\]
Assume that
\begin{equation}\label{example-condition}
  2(\lambda_1+\lambda)>\max\bigl\{
  4\max\{\|a\|_\infty\|B\|_{\cal L(H)},\|m\|_\infty\},
  \ 2\|a\|_\infty\|B\|_{\cal L(H)}+1+2\|m\|_\infty\bigr\},
\end{equation}
where $\lambda_1$ is defined in \eqref{lambda1}.
Lemma \ref{prop:neural-field-example} shows that the hypotheses of
Theorems~\ref{thmR} and~\ref{thm:attractor-convergence} hold. In particular, the averaged equation corresponding to \eqref{eq:app-model} is
\begin{align*}
d\bar u(t)=&\Big[\Delta \bar u(t)-\lambda \bar u(t)-\bar\gamma |\bar u(t)|^2\bar u(t)+\bar a\,B\bar u(t)+\bar b\,h
+\bar m\,\E\bar u(t)\Big]dt
+Q\,dW_t.
\end{align*}
Therefore, by Theorems~\ref{thmT},~\ref{thmR}, and~\ref{thm:attractor-convergence},
we have the following conclusions.
\begin{enumerate}
\item[(i)] Let $u_0\in L^4(\Omega,\cal F_0,\P;V)$, and let $u_\var$ and
$\bar u$ be the solutions of the original and averaged equations with the same
initial value $u_0$. Then for every $T>0$,
\[
\lim_{\var\to0}\E\left(\sup_{0\le t\le T}\|u_\var(t)-\bar u(t)\|^2\right)=0.
\]

\item[(ii)] The original and averaged equations possess 
unique bounded entire mild solutions on $\R$,
still denoted by $u_\var$ and $\bar u$, respectively. They satisfy
\[
\sup_{t\in\R}\E\|u_\var(t)\|^2<\infty,
\qquad
\sup_{t\in\R}\E\|\bar u(t)\|^2<\infty,
\]
and 
\[
\lim_{\var\to0}\sup_{t\in\R}
\E\|u_\var(t)-\bar u(t)\|^2=0.
\]

\item[(iii)] Let $F_0=(\gamma,f,g)$. Then, for every $0<\var\le1$, 
the cocycle generated by \eqref{eq:app-model} admits a pullback attractor
$\{\cal A^\var(F)\}_{F\in\cal H(F_0)}$, while the averaged semigroup admits a
global attractor $\bar{\cal A}$. Moreover, we have
\[
\lim_{\var\to0}\sup_{F\in\cal H(F_0)}
dist\bigl(\cal A^\var(F),\bar{\cal A}\bigr)=0.
\]
\end{enumerate}

\end{exam}

\section{The first and second Bogolyubov theorems}

In this section, we establish Bogolyubov-type theorems for \eqref{eq:SPDEone}. 
We first prove that the solutions of the original equation \eqref{eq:SPDEone} 
converge to those of the averaged equation \eqref{eq:SPDEtwo} on finite time intervals. 
Then, under a stronger monotonicity condition, we prove the second Bogolyubov theorem, 
which concerns the convergence of the corresponding bounded entire mild solutions 
on the whole real line.
To this end, we begin with several lemmas.
\subsection{A priori and continuity estimates}

\begin{lem}\label{lemone}
Assume that \ref{item:H-f} and \ref{item:H-g} (i) hold. Let $u_\var(t)$ be the solution of \eqref{eq:SPDEone} with initial data $u_0$, and $\bar{u}(t)$ be the solution of \eqref{eq:SPDEtwo} with initial data $\bar{u}_0$. For any $p\ge 1$ and $T>0$, if $u_0 \in L^{2p}(\Omega,\cal F_0,\P;H)$, then
\begin{equation*}
\begin{aligned}
&\E\left(\sup_{0\le t\le T}\norm{u_\var (t)}^{2p}\right)
+\E\int_{0}^{T} \norm{u_\var (t)}^{2p-2} \norm{\nabla u_\var (t)}^{2}\,dt\\
&+\E\int_{0}^{T} \gamma\left(\frac{t}{\var}\right)\norm{u_\var (t)}^{2p-2} \norm{u_\var (t)}^{4}_{L^4}\,dt\le C_T(1+\E \norm{u_0}^{2p}),
\end{aligned}
\end{equation*}
where $C_T$ depends on $p,\lambda,K_f,K_g,L_f,L_g$ and $T$. The same result holds for $\bar{u}$ if \ref{item:H-a} holds and $\bar{u}_0 \in L^{2p}(\Omega,\cal F_0,\P;H)$.
\end{lem}
\begin{proof}
By It\^o's formula, we have for any $t\le T$, $p\ge 1$,
\begin{equation}\label{up}
\begin{aligned}
\norm{u_\var(t)}^{2p}
&=\norm{u_0}^{2p}+2p(p-1)\int_{0}^{t}\norm{u_\var(s)}^{2p-4}\norm{g\left(\frac{s}{\var},u_\var(s),\Law{u_\var(s)}\right)^*u_\var(s)}_U^2\,ds\\
&+2p\int_{0}^{t}\norm{u_\var(s)}^{2p-2}\left[\inpro{Au_\var(s)+f\left(\frac{s}{\var},u_\var(s),\Law{u_\var(s)}\right),u_\var(s)}\right.\\
&\left.-\gamma\left(\frac{s}{\var}\right)\inpro{|u_\var(s)|^2u_\var(s),u_\var(s)}
+\frac12\norm{g\left(\frac{s}{\var},u_\var(s),\Law{u_\var(s)}\right)}_{L_2(U,H)}^2\right]\,ds+2pM_t,
\end{aligned}
\end{equation}
where 
\[M_t=\int_{0}^{t}\norm{u_\var(s)}^{2p-2}\inpro{g\left(\frac{s}{\var},u_\var(s),\Law{u_\var(s)}\right)\,dW_s,u_\var(s)}.\]
For every $s\ge0$,
\begin{align}\label{31:1}
\inpro{Au_\var(s),u_\var(s)}=-\norm{\nabla u_\var(s)}^2-\lambda\norm{u_\var(s)}^2\le -(\lambda_1 + \lambda)\norm{u_\var(s)}^2,
\end{align}
and
\begin{align}\label{31:2}
-2\gamma\left(\frac{s}{\var}\right)\inpro{u_\var(s),|u_\var(s)|^2u_\var(s)}=-2\gamma\left(\frac{s}{\var}\right)\norm{u_\var(s)}_{L^4}^4.
\end{align}
It follows from \eqref{eq0731:02} that
\begin{align}\label{31:3}
\norm{u_\var(s)}^{2p-4}\norm{g\left(\frac{s}{\var},u_\var(s),\Law{u_\var(s)}\right)^*u_\var(s)}_U^2
&\le\norm{g\left(\frac{s}{\var},u_\var(s),\Law{u_\var(s)}\right)}_{L_2(U,H)}^2\norm{u_\var(s)}^{2p-2}\notag\\
&\le C\norm{u_\var(s)}^{2p-2}(1+\norm{u_\var(s)}^2+\E\norm{u_\var(s)}^2).
\end{align}
Then by \eqref{eq0731:02-1}, \eqref{up}, \eqref{31:1}, \eqref{31:2} and \eqref{31:3}, we have
\begin{align}\label{31:5}
\norm{u_\var(t)}^{2p}
&\le \norm{u_0}^{2p}
+C\int_{0}^{t}\norm{u_\var(s)}^{2p-2}(1+\norm{u_\var(s)}^2+\E\norm{u_\var(s)}^2)\,ds \notag \\
&-2p\int_{0}^{t}\norm{u_\var(s)}^{2p-2}[\norm{\nabla u_\var(s)}^2+\lambda\norm{u_\var(s)}^2
+\gamma\left(\frac{s}{\var}\right)\norm{u_\var(s)}_{L^4}^4]\,ds+2pM_t.
\end{align}
By the Burkholder--Davis--Gundy and Young inequalities, one sees that
\begin{align}
\label{mar}
\E\left(\sup\limits_{t\in [0,T]}|M_t|\right)
&\le C\E\left(\int_{0}^{T} \norm{u_\var(s)}^{4p-4}\norm{g\left(\frac{s}{\var},u_\var(s),\Law{u_\var(s)}\right)^*u_\var(s)}_U^2\,ds \right)^{\frac{1}{2}} \notag \\
&\le C\E\left(\int_{0}^{T}\norm{u_\var(s)}^{4p-2}(1+\norm{u_\var(s)}^2+\E\norm{u_\var(s)}^2)\,ds\right)^{\frac{1}{2}}\notag\\
&\le C\E\left[\sup_{0\le t\le T}\norm{u_\var(t)}^{2p-1}\left(\int_{0}^{T}(1+\norm{u_\var(s)}^2+\E\norm{u_\var(s)}^2)\,ds\right)^\frac{1}{2}\right]\notag\\
&\le \frac12\E\left(\sup_{0\le t\le T}\norm{u_\var(t)}^{2p}\right)+C\E\left(\int_{0}^{T}(1+\norm{u_\var(s)}^2+\E\norm{u_\var(s)}^2)\,ds\right)^p\notag\\
&\le \frac12\E\left(\sup_{0\le t\le T}\norm{u_\var(t)}^{2p}\right)+C_{T,p}\int_{0}^{T}\left[1+\E\left(\sup_{0\le r\le s}\norm{u_\var(r)}^{2p}\right)\right]\,ds.
\end{align}

By \eqref{31:5}, \eqref{mar}, Young's inequality, and H\"older's inequality, 
we obtain
\begin{align*}
\E\left(\sup_{0\le t\le T}\norm{u_\var(t)}^{2p}\right)
&
\le 2\E{\norm{u_0}^{2p}}
+C\E \sup\limits_{t\in [0,T]}\int_{0}^{t}\norm{u_\var(s)}^{2p-2}(1+\norm{u_\var(s)}^2+\E\norm{u_\var(s)}^2)\,ds\\
&\quad
+C\int_{0}^{T} \left[1+\E\left(\sup\limits_{0\le r\le s}
\norm{u_\var(r)}^{2p}\right)\right]\, ds\\
&
\le 2\E{\norm{u_0}^{2p}}+C\int_{0}^{T} 
\left[1+\E\left(\sup\limits_{0\le r\le s}\norm{u_\var(r)}^{2p}\right)\right]\, ds.
\end{align*}
Gronwall's inequality therefore gives
\begin{equation}\label{eq0617:02}
\E\left(\sup_{0\le t\le T}\|u_\var(t)\|^{2p}\right)\le C_T\left(1+\E\|u_0\|^{2p}\right).
\end{equation}

Note that \eqref{up} yields
\begin{align*}
\|u_\var (t)\|^{2p}
&+\int_{0}^{t}\|u_\var(s)\|^{2p-2}\|\nabla u_\var(s)\|^2\,ds
+\int_{0}^{t}\gamma\left(\frac{s}{\var}\right)\|u_\var(s)\|^{2p-2}\|u_\var(s)\|_{L^4}^4\,ds\notag \\
&\le \|u_0\|^{2p}
+C\int_{0}^{t}\|u_\var(s)\|^{2p-2}(1+\|u_\var(s)\|^2+\E\|u_\var(s)\|^2)\,ds
+2p M_t.
\end{align*}
Hence,
\begin{align}\label{eq0617:01}
&\E\int_{0}^{T}\|u_\var(s)\|^{2p-2}\|\nabla u_\var(s)\|^2\,ds
+\E\int_{0}^{T}\gamma\left(\frac{s}{\var}\right)
\|u_\var(s)\|^{2p-2}\|u_\var(s)\|_{L^4}^4\,ds \\\nonumber
&\le \E\|u_0\|^{2p}
+C\int_{0}^{T}
\E\Big[
\|u_\var(s)\|^{2p-2}
\bigl(1+\|u_\var(s)\|^2+\E\|u_\var(s)\|^2\bigr)
\Big]\,ds.
\end{align}
Combining \eqref{eq0617:01}, \eqref{eq0617:02} and Young's inequality, we obtain
\[\E\int_{0}^{T}\|u_\var(s)\|^{2p-2}\|\nabla u_\var(s)\|^2\,ds
+\E\int_{0}^{T}\gamma\left(\frac{s}{\var}\right)
\|u_\var(s)\|^{2p-2}\|u_\var(s)\|_{L^4}^4\,ds
\le C_T\bigl(1+\E\|u_0\|^{2p}\bigr),\]
Together with \eqref{eq0617:02}, the proof is complete.
\end{proof}

The following estimate is proved by Galerkin approximation in the appendix.
\begin{lem}\label{lemtwo}
Assume that \ref{item:H-f} and \ref{item:H-g} hold. Let $u_\var(t)$ be the solution of \eqref{eq:SPDEone} with initial data $u_0$, and $\bar{u}(t)$ be the solution of \eqref{eq:SPDEtwo} with initial data $\bar{u}_0$. If $u_0 \in L^{2p}(\Omega,\cal{F}_0,\P;V)$, then for any $p\ge1$ and $T>0$,
\[
\E\left(\sup_{0\le t\le T}\norm{u_\var (t)}^{2p}_V\right)
\le C_T(1+\E{\norm{u_0}_V^{2p}}).
\]
If, in addition, \ref{item:H-a} holds, then the same estimate holds for $\bar{u}$ whenever $\bar{u}_0\in L^{2p}(\Omega,\cal{F}_0,\P;V)$.
\end{lem}

Since $\|\cdot\|_V$ and the usual $H^1$-norm are equivalent,
\begin{align}
\label{Hnorm}
\sup_{0\le t\le T}\E \norm{u_\var(t)}^{2p}_{H^1}\le C_T(1+\E{\norm{u_0}_{H^1}^{2p}}).
\end{align}

\begin{lem}\label{Law-con}
Assume that \ref{item:H-f} and \ref{item:H-g} hold. For any
\(\zeta,\eta\in L^2(\Omega,\cal F_0,\P;H)\), let
\(u_\var(t,\zeta):=u_\var(t;0,\zeta)\) and
\(u_\var(t,\eta):=u_\var(t;0,\eta)\) denote the solutions of
\eqref{eq:SPDEone} starting at time \(0\) from \(\zeta\) and \(\eta\),
respectively. Then we have
\begin{align*}
\E\norm{u_\var(t,\zeta)-u_\var(t,\eta)}^2\le e^{\beta t}\E\norm{\zeta-\eta}^2,
\end{align*}
where $\beta:=-2\lambda_1-2\lambda+4L_f+4L_g^2$.
Moreover, for any \(u_0\in L^4(\Omega,\P;V)\) and $0\le t_1\le t_2\le T$, we have
\begin{equation}\label{eq0620:06}
\E\norm{u_\var(t_2)-u_\var(t_1)}^2\le C(t_2-t_1)e^{C(t_2-t_1)}\le C_T(t_2-t_1).
\end{equation}

If \ref{item:H-a} also holds, the same estimates apply to
\eqref{eq:SPDEtwo}.
\end{lem}
\begin{proof}
Let $z(t)=u_1(t)-u_2(t)=u_\var(t,\zeta)-u_\var(t,\eta)$. By It\^o's formula, 
\ref{item:H-f} and \ref{item:H-g}, we have
\begin{align*}
\frac{d}{dt}\E\norm{z(t)}^2
&
=2\E\inpro{Az(t),z(t)}-2\gamma\left(\frac{t}{\var}\right)
\E \inpro{|u_1(t)|^2u_1(t)-|u_2(t)|^2u_2(t),z(t)} \notag\\
&\quad
+2\E\inpro{f\left(\frac{t}{\var}, u_1(t), \Law{u_1(t)}\right)-f\left(\frac{t}{\var}, u_2(t), \Law{u_2(t)}\right),z(t)}\notag\\
&\quad
+\E\norm{g\left(\frac{t}{\var}, u_1(t), \Law{u_1(t)}\right)-g\left(\frac{t}{\var}, u_2(t), \Law{u_2(t)}\right)}_{L_2(U,H)}^2\notag\\
&
\le \left(-2\lambda_1-2\lambda+4L_f+4L_g^2\right)\E\norm{z(t)}^2.
\end{align*}
Gronwall's inequality therefore gives
\begin{align*}
  \E\norm{z(t)}^2\le e^{\left(-2\lambda_1-2\lambda+4L_f+4L_g^2\right)t}\E\norm{\zeta-\eta}^2.
  \end{align*}

Fix $t_1\ge0$. By It\^o's formula, one sees that
\begin{align}\label{yone}
\E\norm{u_\var(t_2)-u_\var(t_1)}^2
&
=2\E\int_{t_1}^{t_2}\inpro{A u_\var(r),u_\var(r)-u_\var(t_1)}\,dr\notag\\
&\quad
-2\E\int_{t_1}^{t_2}\gamma\left(\frac{r}{\var}\right)
\inpro{|u_\var(r)|^2u_\var(r),u_\var(r)-u_\var(t_1)}\,dr\notag\\
&\quad
+2\E\int_{t_1}^{t_2}\inpro{f\left(\frac{r}{\var}, u_\var(r), 
\Law{u_\var(r)}\right),u_\var(r)-u_\var(t_1)}\,dr\notag\\
&\quad
+\E\int_{t_1}^{t_2}\norm{g\left(\frac{r}{\var}, u_\var(r), \Law{u_\var(r)}\right)}_{L_2(U,H)}^2\,dr.
\end{align}
For the first term, it follows from Lemma~\ref{lemtwo} that
\begin{align}\label{ytwo}
&
2\E\int_{t_1}^{t_2}\inpro{A u_\var(r),u_\var(r)-u_\var(t_1)}\,dr\notag\\
&
=\E\int_{t_1}^{t_2}\left(-2\norm{\nabla u_\var(r)}^2+2\inpro{\nabla u_\var(r),\nabla{u_\var(t_1)}}-2\lambda\norm{u_\var(r)}^2+2\lambda\inpro{u_\var(r),u_\var(t_1)}\right)
dr\notag\\
&
\le \E\int_{t_1}^{t_2}\left(-2\norm{\nabla u_\var(r)}^2+\norm{\nabla u_\var(r)}^2+\norm{\nabla u_\var(t_1)}^2+\lambda\norm{u_\var(t_1)}^2\right)dr\notag\\
&
\le \int_{t_1}^{t_2}\E\norm{u_\var(t_1)}_V^2 dr\notag\\
&
\leq C(1+\E\|u_0\|_V^2)(t_2-t_1).
\end{align}
For the cubic term, H\"older's inequality gives
\begin{align}\label{ythree}
-2\E\int_{t_1}^{t_2}\gamma\left(\frac{r}{\var}\right)
\inpro{|u_\var(r)|^2u_\var(r),u_\var(r)-u_\var(t_1)}dr
&
\le C\E\int_{t_1}^{t_2} (\norm{u_\var(r)}^4_{L^4}+\norm{u_\var(t_1)}^4_{L^4})dr\notag\\
&
\leq C(1+\E\|u_0\|_{V}^4)(t_2-t_1).
\end{align}
By \eqref{eq0731:02-1}, \eqref{eq0731:02}, and Lemma~\ref{lemone},
we obtain
\begin{align*}
&
2\E\int_{t_1}^{t_2}\inpro{f\left(\frac{r}{\var}, u_\var(r), \Law{u_\var(r)}\right),
u_\var(r)-u_\var(t_1)}\,dr
+\E\int_{t_1}^{t_2}\norm{g\left(\frac{r}{\var}, u_\var(r), \Law{u_\var(r)}\right)}_{L_2(U,H)}^2\,dr\\
&
\leq \E\int_{t_1}^{t_2}\|u_\var(r)-u_\var(t_1)\|^2dr+
\E\int_{t_1}^{t_2}C\left(1+\norm{u_\var(r)}^2+\E\norm{u_\var(r)}^2\right)dr\\
&
\leq \E\int_{t_1}^{t_2}\|u_\var(r)-u_\var(t_1)\|^2dr+C(1+\E\|u_0\|^2)(t_2-t_1),
\end{align*}
which along with \eqref{yone}, \eqref{ytwo}, and \eqref{ythree} gives
\begin{align*}
\E\norm{u_\var(t_2)-u_\var(t_1)}^2
&
\le \int_{t_1}^{t_2}C(1+\E\norm{u_\var(r)-u_\var(t_1)}^2)\,dr\\
&
\le \int_{t_1}^{t_2}C\E\norm{u_\var(r)-u_\var(t_1)}^2\,dr+C(t_2-t_1).
\end{align*}
Hence, by Gronwall's inequality, we obtain
\[
\E\norm{u_\var(t_2)-u_\var(t_1)}^2\le C(t_2-t_1)e^{C(t_2-t_1)}\le C_T(t_2-t_1).
\]

The proof for the averaged equation is similar.
\end{proof}

\subsection{Proof of the first Bogolyubov theorem}\label{secFBT}

We use the following analytic-semigroup and heat-kernel estimates; see
\cite[Theorem 6.13]{Pazy83} for the first and the standard heat-kernel
bound on the torus for the second.
\begin{lem}
Let $\{S(t)\}_{t\ge0}$ be the analytic semigroup generated by $A$ on $H$. 
Then we have the following conclusions.
\begin{itemize}
\item[(i)]
For any $\theta\in(0,1]$, there exists a constant $C_\theta>0$ such that
\begin{equation}\label{eq0620:02}
\norm{(S(h)-I)S(r)}_{\cal{L}(H,H)}\le C_\theta h^\theta r^{-\theta},\qquad \forall\, h>0,\ r>0.
\end{equation}
\item[(ii)] 
There exist $C,\omega>0$ such that for any $1\leq p\leq q\leq \infty$,  
\begin{equation}\label{eq0620:01}
\|S(t)\|_{\cal{L}(L^p,L^q)}\le C\,t^{-\frac d2(\frac1p-\frac1q)}e^{-\omega t}.
\end{equation}
\end{itemize}
\end{lem}

Fix $\delta$ sufficiently small. 
For a given process $\varphi$, we define the step process $\tilde \varphi$ by
$\tilde \varphi(t)=\varphi(k\delta)$,
$t\in [k\delta,(k+1)\delta)\cap[0,T]$,
where $k=0,1,\ldots,\lfloor T/\delta\rfloor$. 
The next lemma controls the error caused by this time discretization.
\begin{lem}\label{lem0628:01}
Assume that \ref{item:H-f}, \ref{item:H-g} and \ref{item:H-a} hold and
\(u_0\in L^4(\Omega,\cal F_0,\P;V)\).
Fix $R>0$, and define the stopping time
\[
\tau_R:=\inf\{t\ge0:\|u_\var(t)\|_V+\|\bar{u}(t)\|_V\ge R\},
\]
with the convention $\inf\varnothing=\infty$.
We have
\begin{equation*}
\E\left(\sup_{0\leq t\leq T\wedge\tau_R}\left\|\int_0^t S(t-s)\left(-\gamma(s/\var)+\bar{\gamma}\right)
\Big(|u_\var(s)|^2u_\var(s)-|\widetilde{u}_\var(s)|^2
\widetilde{u}_\var(s)\Big)\,ds\right\|^2\right)
\leq C_{T,R}\delta^{\frac14}.
\end{equation*}
\end{lem}
\begin{proof}
Set
\[
R_\var(t):=\int_0^t S(t-s)\left(-\gamma(s/\var)+\bar{\gamma}\right)
\Big(|u_\var(s)|^2u_\var(s)-|\widetilde{u}_\var(s)|^2
\widetilde{u}_\var(s)\Big)\,ds.
\]
By \eqref{eq0620:01} with \(p=\frac43\), \(q=2\), for any \(t\in[0,T\wedge\tau_R]\),
\begin{align}\label{eq0620:04}
\|R_\var(t)\|
&
\le 2\|\gamma\|_\infty \int_0^t\|S(t-s)\big(|u_\var(s)|^2u_\var(s)
-|\widetilde{u}_{\var}(s)|^2\widetilde{u}_{\var}(s)\big)\|\,ds \notag\\
&
\le C \int_0^t (t-s)^{-\frac d8}
\||u_\var(s)|^2u_\var(s)-|\widetilde{u}_{\var}(s)|^2\widetilde{u}_{\var}(s)\|_{L^{\frac43}}\,ds.
\end{align}
Moreover, using the pointwise inequality
\begin{equation}\label{eq0628:04}
\bigl||a|^2a-|b|^2b\bigr|\le C\bigl(|a|^2+|b|^2\bigr)|a-b|,
\end{equation}
together with H\"older's inequality, we obtain
\begin{align}\label{eq0620:05}
\||u_\var(s)|^2u_\var(s)-|\widetilde{u}_{\var}(s)|^2\widetilde{u}_{\var}(s)\|_{L^{\frac43}}
&
\le C\big\|\big(|u_\var(s)|^2+|\widetilde{u}_{\var}(s)|^2\big)
\left|u_\var(s)-\widetilde{u}_{\var}(s)\right|\big\|_{L^{\frac43}}\notag\\
&
\le C\big(\|u_\var(s)\|_{L^4}^2+\|\widetilde{u}_{\var}(s)\|_{L^4}^2\big)
\|u_\var(s)-\widetilde{u}_{\var}(s)\|_{L^4}.
\end{align}
Therefore, by \eqref{eq0620:04}, \eqref{eq0620:05} and
the fact that $d\leq 3$, we have
\begin{align*}
\|R_\var(t)\|^2
&
\le C \left(\int_0^t (t-s)^{-d/8}
\big(\|u_\var(s)\|_{L^4}^2+\|\widetilde{u}_{\var}(s)\|_{L^4}^2\big)
\|u_\var(s)-\widetilde{u}_{\var}(s)\|_{L^4}\,ds\right)^2\\
&
\leq C\left(\int_0^t (t-s)^{-d/4}\,ds\right)\left(\int_0^t
\big(\|u_\var(s)\|_{L^4}^2
+\|\widetilde{u}_{\var}(s)\|_{L^4}^2\big)^2
\|u_\var(s)-\widetilde{u}_{\var}(s)\|_{L^4}^2\,ds\right)\\
&
\leq C_T\int_0^t\left(\|u_\var(s)\|_{L^4}^2
+\|\widetilde{u}_{\var}(s)\|_{L^4}^2\right)^2
\|u_\var(s)-\widetilde{u}_{\var}(s)\|_{L^4}^2\,ds.
\end{align*}
Consequently, the Sobolev embedding and \eqref{eq0620:06} give
\begin{align*}
&
\E\left(\sup_{0\le t\le T\wedge\tau_R}\|R_\var(t)\|^2\right)\notag\\
&
\le C_T\E\left(\int_0^{T\wedge\tau_R}
\big(\|u_\var(s)\|_{L^4}^4+\|\widetilde{u}_{\var}(s)\|_{L^4}^4\big)
\|u_\var(s)-\widetilde{u}_{\var}(s)\|_{L^4}^2\,ds\right)\notag\\
&
\leq C_T
\E\left(\int_0^{T\wedge\tau_R}
\big(\|u_\var(s)\|_{V}^4+\|\widetilde{u}_{\var}(s)\|_{V}^4\big)
\|u_\var(s)-\widetilde{u}_{\var}(s)\|_{H^1}^{\frac32}
\|u_\var(s)-\widetilde{u}_{\var}(s)\|_{H}^{\frac12}\,ds\right)\notag\\
&
\leq C_{R,T}\left(\int_0^{T}
\E\|u_\var(s)-\widetilde{u}_{\var}(s)\|_{H}^2\,ds\right)^{\frac14}
\leq C_{R,T}\delta^{\frac14}.
\end{align*}
\end{proof}

\begin{proof}[Proof of Theorem~\ref{thmT}]
Let $\widetilde u_\var$ and $\tilde u:=\widetilde{\bar u}$ denote the step
approximations defined above.
By Lemma~\ref{lem:mild-solution}, we have
\begin{align}\label{22:1}
u_\var(t)-\bar{u}(t)
&
=\int_{0}^{t}S(t-s)\left[f\left(\frac{s}{\var},u_\var(s),\Law{u_\var(s)}\right)
-\bar{f}\left(\bar{u}(s), \Law{\bar{u}(s)}\right)\right]\,ds\notag\\
&\qquad
+\int_{0}^{t}S(t-s)\left[-\gamma(\frac{s}{\var})|u_\var(s)|^2 u_\var(s)+\bar{\gamma}|\bar{u}(s)|^2 \bar{u}(s)\right]\,ds\notag\\
&\qquad
+\int_{0}^{t}{S(t-s)\left[g\left(\frac{s}{\var}, u_\var(s), \Law{u_\var(s)}\right)-\bar{g}\left(\bar{u}(s), \Law{\bar{u}(s)}\right)\right]\,dW_s}\notag\\
&
=:\cal{I}_1(t)+\cal{I}_2(t)+M_t,
\end{align}
where 
\[M_t=\int_{0}^{t}{S(t-s)\left(g\left(\frac{s}{\var}, u_\var(s), \Law{u_\var(s)}\right)-\bar{g}\left(\bar{u}(s), \Law{\bar{u}(s)}\right)\right)\,dW_s}.\]
  
First, we handle the term $\cal{I}_1(t)$,
\begin{align}\label{22:2}
\cal{I}_1(t)
&
=\int_{0}^{t}S(t-s)\left[f\left(\frac{s}{\var},u_\var(s),\Law{u_\var(s)}\right)
-f\left(\frac{s}{\var},\bar{u}(s), \Law{\bar{u}(s)}\right)\right]\,ds\notag\\
&\qquad
+\int_{0}^{t}S(t-s)\left[f\left(\frac{s}{\var},\bar{u}(s), \Law{\bar{u}(s)}\right)
-\bar{f}\left(\bar{u}(s), \Law{\bar{u}(s)}\right)\right]\,ds \notag\\
&
=:\cal{I}_{11}(t)+\cal{I}_{12}(t).
\end{align}
Applying H\"older's inequality and \ref{item:H-f}, we obtain
\begin{align}\label{Ioneone} 
\E\left(\sup_{0\le t\le T}\norm{\cal{I}_{11}(t)}^2\right)
&
\le T\E\int_{0}^{T} \norm{f\left(\frac{s}{\var},u_\var(s),\Law{u_\var(s)}\right)
-f\left(\frac{s}{\var},\bar{u}(s),\Law{\bar{u}(s)}\right)}^2\,ds  \notag \\ 
&
\le T\E\int_{0}^{T}\left(L_f\left[\norm{u_\var(s)-\bar{u}(s)}
+W_2(\Law{u_\var(s)},\Law{\bar{u}(s)})\right]\right)^2 \,ds\notag \\
&
\le C_T\int_{0}^{T}\E\left(\sup_{0\le r\le s}\norm{u_\var(r)-\bar{u}(r)}^2\right)\,ds.  
\end{align}

For $\cal{I}_{12}(t)$, by H\"older's inequality, \ref{item:H-f},
Remark \ref{rem0618} and Lemma~\ref{Law-con}, 
we have
\begin{align}\label{ftwo}
\E\left(\sup_{0\leq t\leq T}\|\cal{I}_{12}(t)\|^2\right)
&
=\E\left(\sup_{0\leq t\leq T}
\left\|\int_{0}^{t}S(t-s)\left[f\left(\frac{s}{\var},\bar{u}(s), \Law{\bar{u}(s)}\right)
-\bar{f}\left(\bar{u}(s), \Law{\bar{u}(s)}\right)\right]\,ds\right\|^2\right)\notag\\
&
\leq \E\left(\sup_{0\leq t\leq T}
\left\|\int_{0}^{t}S(t-s)\left[f\left(\frac{s}{\var},\bar{u}(s), \Law{\bar{u}(s)}\right)
-f\left(\frac{s}{\var},\tilde{u}(s), \Law{\tilde{u}(s)}\right)\right]\,ds\right\|^2\right) \notag  \\
&\quad
+\E\left(\sup_{0\leq t\leq T}
\left\|\int_{0}^{t}S(t-s)\left[f\left(\frac{s}{\var},\tilde{u}(s), \Law{\tilde{u}(s)}\right)
-\bar{f}\left(\tilde{u}(s),\Law{\tilde{u}(s)}\right)\right]\,ds\right\|^2\right) \notag  \\
&\quad
+\E\left(\sup_{0\leq t\leq T}
\left\|\int_{0}^{t}S(t-s)\left[\bar{f}\left(\tilde{u}(s),\Law{\tilde{u}(s)}\right)
-\bar{f}\left(\bar{u}(s), \Law{\bar{u}(s)}\right)\right]\,ds\right\|^2\right) \notag  \\
&
\leq C\int_{0}^{T}\E\|\bar{u}(s)-\tilde{u}(s)\|^2ds+\cal{I}_{12}^2 \notag  \\
&
\leq C_T\delta+\cal{I}_{12}^2,
\end{align}
where
\[
\cal I_{12}^2:=\E\left(\sup_{0\le t\le T}\left\|\int_0^tS(t-s)
\left[f\left(\frac{s}{\var},\tilde u(s),\Law{\tilde u(s)}\right)
-\bar f\left(\tilde u(s),\Law{\tilde u(s)}\right)\right]\,ds\right\|^2\right).
\]
For $k=0,1,\ldots,\lfloor T/\delta\rfloor$, set
\[
\cA_k(s):=f\left(\frac{s}{\var},\bar u(k\delta),\Law{\bar u(k\delta)}\right)
-\bar f\left(\bar u(k\delta),\Law{\bar u(k\delta)}\right),
\quad s\in[k\delta,(k+1)\delta).
\]
Hence,
\begin{align}\label{eq0619:02}
\cal I_{12}^2&\le3\cal J_1+3\cal J_2+3\cal J_3,
\end{align}
where
\begin{align*}
\cal J_1&:=\E\sup_{0\le t\le T}\left\|\sum_{k=0}^{\lfloor t/\delta\rfloor-2}S(t-(k+1)\delta)
\int_{k\delta}^{(k+1)\delta}\cA_k(s)\,ds\right\|^2,\\
\cal J_2&:=\E\sup_{0\le t\le T}\left\|\sum_{k=0}^{\lfloor t/\delta\rfloor-2}\int_{k\delta}^{(k+1)\delta}
\bigl(S(t-s)-S(t-(k+1)\delta)\bigr)\cA_k(s)\,ds\right\|^2,\\
\cal J_3&:=\E\sup_{0\le t\le T}\left\|\int_{r_t}^{t}S(t-s)
\left(f\left(\frac{s}{\var},\tilde u(s),\Law{\tilde u(s)}\right)
-\bar f\left(\tilde u(s),\Law{\tilde u(s)}\right)\right)\,ds\right\|^2,
\end{align*}
where $r_t:=\max\{0,(\lfloor t/\delta\rfloor-1)\delta\}$.
Remark~\ref{rem0618} gives
\begin{align}\label{eq0618:02}
\|\cA_k(s)\|
&\le C\left(1+\sup_{0\le t\le T}\|\bar u(t)\|
+\sup_{0\le t\le T}\bigl(\E\|\bar u(t)\|^2\bigr)^{1/2}\right).
\end{align}
The same estimate holds for $f\left(\frac{s}{\var},\tilde u(s),\Law{\tilde u(s)}\right)
-\bar f\left(\tilde u(s),\Law{\tilde u(s)}\right)$. Hence, by H\"older's inequality and Lemma~\ref{lemone},
\begin{align}\label{eq0618:01}
\cal J_3\le C_T\delta^2\le C_T\delta.
\end{align}
For $\cal J_1$, the contraction property of $S$, a change of variables, \ref{item:H-a}, and Lemma~\ref{lemone} yield
\begin{align}\label{eq0619:03}
\cal J_1
&\le\E\sup_{0\le t\le T}\left(\sum_{k=0}^{\lfloor t/\delta\rfloor-2}
\left\|\int_{k\delta}^{(k+1)\delta}\cA_k(s)\,ds\right\|\right)^2\notag\\
&\le\E\sup_{0\le t\le T}\left(\sum_{k=0}^{\lfloor t/\delta\rfloor-2}
\delta\omega^f\left(\frac{\delta}{\var}\right)
\left(1+\|\bar u(k\delta)\|+\wts{\bar u(k\delta)}\right)\right)^2\notag\\
&\le C_T\omega^f\left(\frac{\delta}{\var}\right)^2.
\end{align}

By \eqref{eq0620:02}, for any $\theta\in(0,1)$ there exists $C_\theta>0$
such that for all $s\in[k\delta,(k+1)\delta)$ and $k=0,1,\ldots,\lfloor t/\delta\rfloor-2$,
\begin{align}\label{22:5}
\norm{S(t-s)-S(t-(k+1)\delta)}_{\cal{L}(H,H)}
&
=\norm{(S((k+1)\delta-s)-I)S(t-(k+1)\delta)}_{\cal{L}(H,H)}\notag\\
&
\le C_\theta((k+1)\delta-s)^{\theta}(t-(k+1)\delta)^{-\theta}.
\end{align}
For $\lfloor t/\delta\rfloor\ge2$,
\begin{equation}\label{eq0619:01}
\sum_{j=1}^{\lfloor t/\delta\rfloor-1}j^{-\theta}
\le 1+\int_{1}^{\lfloor t/\delta\rfloor-1}x^{-\theta}\,dx
=1+\frac{(\lfloor t/\delta\rfloor-1)^{1-\theta}-1}{1-\theta}
\leq \frac{(\lfloor t/\delta\rfloor-1)^{1-\theta}}{1-\theta}.
\end{equation}
Equations~\eqref{22:5}, \eqref{eq0618:02}, and \eqref{eq0619:01}, together
with Lemma~\ref{lemone}, yield
\begin{align*}
\mathcal J_2
&\le C_{T,\theta}\sup_{0\le t\le T}
\left[\delta^{1+\theta}\sum_{k=0}^{\lfloor t/\delta\rfloor-2}
\bigl(t-(k+1)\delta\bigr)^{-\theta}\right]^2\\
&\le C_{T,\theta}\delta^2\sup_{0\le t\le T}
\left(\sum_{j=1}^{\lfloor t/\delta\rfloor-1}j^{-\theta}\right)^2
\le C_{T,\theta}\delta^{2\theta}.
\end{align*}
This, together with \eqref{ftwo}, \eqref{eq0619:02}, \eqref{eq0618:01}, 
and \eqref{eq0619:03} implies that
\begin{align}\label{Ionetwo}
\E\left(\sup_{0\leq t\leq T}\|\cal{I}_{12}(t)\|^2\right)
\leq C_{\theta,T}\left(\delta+{\omega^f(\frac{\delta}{\var})}^2
+\delta^{2\theta}\right).
\end{align}
Substituting \eqref{Ioneone} and \eqref{Ionetwo} into \eqref{22:2}, we obtain for all $\theta\in(0,1)$,
\begin{align}\label{Ione}
\E\left(\sup_{0\le t\le T}\norm{\cal{I}_{1}(t)}^2\right)
\le C\left(\int_{0}^{T}\E\left(\sup_{0\le r\le s}\norm{u_\var(r)-\bar{u}(r)}^2\right)\,ds
+\delta^{2\theta}+\delta+{\omega^f(\frac{\delta}{\var})}^2\right).
\end{align}
 
Note that
\begin{align}\label{22:6}
\cal I_2(t)
&
=\int_{0}^{t}S(t-s)\left(-\gamma\left(\frac{s}{\var}\right)+\bar{\gamma}\right)\Big(|u_\var(s)|^2 u_\var(s)-|\widetilde{u}_{\var}(s)|^2\widetilde{u}_{\var}(s)\Big)\,ds\notag\\
&\quad
+\int_{0}^{t}S(t-s)\left(-\gamma\left(\frac{s}{\var}\right)+\bar{\gamma}\right)|\widetilde{u}_{\var}(s)|^2\widetilde{u}_{\var}(s)\,ds\notag\\
&\quad
+\int_{0}^{t}S(t-s)(-\bar{\gamma})[|u_{\var}(s)|^2u_{\var}(s)-|\bar{u}(s)|^2 \bar{u}(s)]\,ds\notag\\
&
=:B_1(t)+B_2(t)+B_3(t).
\end{align}
Lemma~\ref{lem0628:01} gives
\begin{align}\label{22:7}
\E\left(\sup_{0\le t\le T\wedge\tau_R}\|B_1(t)\|^2\right)
\le C_{R,T}\delta^{\frac14}.
\end{align}
The estimate for \(B_2(t)\) is the same freezing argument as for
\(\cal I_{12}^2\), with \eqref{wgamma} in place of \eqref{wf}. Thus,
for every \(\theta\in(0,1)\),
\begin{align}\label{Btwo}
\E\left(\sup_{0\le t\le T\wedge\tau_R}\norm{B_{2}(t)}^2\right)
\le C_{T,R,\theta}\left(\delta^{2\theta}+\delta+\omega^\gamma\left(\frac{\delta}{\var}\right)^2\right).
\end{align}

It remains to estimate \(B_3(t)\).
By \eqref{eq0620:01}, \eqref{eq0628:04}, and the
embedding \(H^1(\T^d)\hookrightarrow L^6(\T^d)\), we obtain
\begin{align*}
\|B_{3}(t)\|
&\le C\int_0^t
\left\|S(t-s)\left(|u_\var(s)|^2u_\var(s)-|\bar u(s)|^2\bar u(s)\right)
\right\|\,ds\\
&\le C\int_0^t(t-s)^{-\frac d2(\frac56-\frac12)}
\||u_\var(s)|^2u_\var(s)-|\bar u(s)|^2\bar u(s)\|_{L^{6/5}}\,ds\\
&\leq C\int_0^t(t-s)^{-\frac d2(\frac56-\frac12)}
\left(\|u_\var(s)\|_{L^6}^2+\|\bar u(s)\|_{L^6}^2\right)\|u_\var(s)-\bar u(s)\|\,ds\\
&\leq C\int_0^t(t-s)^{-\frac d6}
\left(\|u_\var(s)\|_{H^1}^2+\|\bar u(s)\|_{H^1}^2\right)\|u_\var(s)-\bar u(s)\|\,ds.
\end{align*}
Since $d\leq 3$,
\begin{equation}\label{eq0628:03}
\E\left(\sup_{0\le r\le t\wedge\tau_R}\|B_3(r)\|^2\right)
\le C_{R,T}\int_0^t(t-s)^{-1/2}
\E\left(\sup_{0\le r\le s\wedge\tau_R}\|u_\var(r)-\bar u(r)\|^2\right)\,ds.
\end{equation}
Combining \eqref{22:6}, \eqref{22:7}, \eqref{Btwo}, and
\eqref{eq0628:03} yields
\begin{align}\label{Ithree}
\E\left(\sup_{0\le r\le t\wedge\tau_R}\norm{\cal I_2(r)}^2\right)
&\le C_{T,R,\theta}\left(\delta^{\frac14}+\delta^{2\theta}
+\omega^\gamma\left(\frac{\delta}{\var}\right)^2\right)\notag\\
&\quad
+C_{R,T}\int_0^t(t-s)^{-\frac12}
\E\left(\sup_{0\le r\le s\wedge\tau_R}\|u_\var(r)-\bar u(r)\|^2\right)
\,ds.
\end{align}

Since \(S(t)\) is a contraction semigroup on \(H\), the maximal inequality
for stochastic convolutions generated by contraction semigroups gives
\begin{align*}
\E\left(\sup_{0\le t\le T}\|M_t\|^2\right)
&
=\E\left(\sup_{0\le t\le T}\left\|\int_0^rS(r-s)\left(g\left(\frac{s}{\var},u_\var(s),\Law{u_\var(s)}\right)
-\bar g\left(\bar u(s),\Law{\bar u(s)}\right)\right)\,dW_s\right\|^2\right)\\
&\le C\,\E\int_0^T\|g\left(\frac{s}{\var},u_\var(s),\Law{u_\var(s)}\right)
-\bar g\left(\bar u(s),\Law{\bar u(s)}\right)\|_{L_2(U,H)}^2\,ds,
\end{align*}
Together with \ref{item:H-g} (i), Remark~\ref{rem0618}, and \eqref{eq0620:06}, this implies that
\begin{align}\label{eq0630:02}
\E\left(\sup_{0\le r\le T}\|M_r\|^2\right)
&\le C\E\int_0^T\|g\left(\frac{s}{\var},u_\var(s),\Law{u_\var(s)}\right)
-g\left(\frac{s}{\var},\bar u(s),\Law{\bar u(s)}\right)\|_{L_2(U,H)}^2\,ds\notag\\
&\quad
+C\E\int_0^T\|g\left(\frac{s}{\var},\bar u(s),\Law{\bar u(s)}\right)
-g\left(\frac{s}{\var},\tilde u(s),\Law{\tilde u(s)}\right)\|_{L_2(U,H)}^2\,ds\notag\\
&\quad
+C\E\int_0^T\|g\left(\frac{s}{\var},\tilde u(s),\Law{\tilde u(s)}\right)
-\bar g\left(\tilde u(s),\Law{\tilde u(s)}\right)\|_{L_2(U,H)}^2\,ds\notag\\
&\quad
+C\E\int_0^T\|\bar{g}\left(\tilde u(s),\Law{\tilde u(s)}\right)
-\bar g\left(\bar u(s),\Law{\bar u(s)}\right)\|_{L_2(U,H)}^2\,ds\notag\\
&
\leq C\int_0^T\left(\E\|u_\var(s)-\bar{u}(s)\|^2
+\E\|\tilde u(s)-\bar{u}(s)\|^2 \right) ds+\mathcal J\notag\\
&
\leq C\int_0^T\E\|u_\var(s)-\bar{u}(s)\|^2\,ds+C_T\delta+\mathcal J,
\end{align}
where
\[
\mathcal J:=C\E\int_0^T\|g\left(\frac{s}{\var},\tilde u(s),\Law{\tilde u(s)}\right)
-\bar g\left(\tilde u(s),\Law{\tilde u(s)}\right)\|_{L_2(U,H)}^2\,ds.
\]
As in the estimate of $\mathcal I_{12}^2$, splitting the time interval into
blocks and using \eqref{wg} gives
\begin{equation}\label{eq0630:01}
\mathcal J\leq C_T\left[\delta+\omega^g\left(\frac{\delta}{\var}\right)\right].
\end{equation}
Combining \eqref{eq0630:02} and \eqref{eq0630:01}, we obtain
\begin{equation}\label{eq0630:03}
\E\left(\sup_{0\le r\le t}\|M_r\|^2\right)
\leq C\int_0^t \E\|u_\var(s)-\bar{u}(s)\|^2\, ds
+C_T\left(\delta+\omega^g\left(\frac{\delta}{\var}\right)\right).
\end{equation}

Let $\delta=\sqrt{\var}$. It follows from \eqref{22:1}, \eqref{Ithree}, \eqref{eq0630:03} 
and \eqref{Ione} with $\theta=\frac12$ that
\begin{align}\label{eq0630:04}
\E\left(\sup_{0\le r\le t\wedge\tau_R}\|u_\var(r)-\bar u(r)\|^2\right)
&\leq C_{T,R}\left[\delta^{\frac14}+\omega^f\left(\frac{\delta}{\var}\right)^2
+\omega^\gamma\left(\frac{\delta}{\var}\right)^2
+\omega^g\left(\frac{\delta}{\var}\right)\right]\notag\\
&\quad
+C_{R,T}\int_0^t\left((t-s)^{-\frac12}+1\right)
\E\left(\sup_{0\le q\le s}\|u_\var(q)-\bar u(q)\|^2\right)\,ds\notag\\
&\leq C_{T,R}\left[\var^{\frac18}+\omega^f\left(\var^{-\frac12}\right)^2
+\omega^\gamma\left(\var^{-\frac12}\right)^2+\omega^g\left(\var^{-\frac12}\right)\right]\notag\\
&\quad
+C_{R,T}\int_0^t(t-s)^{-\frac12}
\E\left(\sup_{0\le q\le s}\|u_\var(q)-\bar u(q)\|^2\right)\,ds.
\end{align}
By Chebyshev's inequality and Lemma~\ref{lemtwo} with $p=2$, one sees that
\[
\P(\tau_R\le T)\le \frac{C_T}{R^4}.
\]
Consequently, Cauchy--Schwarz and Lemma~\ref{lemone} with $p=2$ give
\begin{align}\label{eq0630:05}
\E\left(\mathbf 1_{\{\tau_R\le t\}}\sup_{0\le r\le t}
\|u_\var(r)-\bar u(r)\|^2\right)
&\le \P(\tau_R\le T)^{1/2}
\left(\E\sup_{0\le r\le T}\|u_\var(r)-\bar u(r)\|^4\right)^{1/2}\notag\\
&\le C_T R^{-2}.
\end{align}
Combining \eqref{eq0630:04} and \eqref{eq0630:05}, we find for $0\le t\le T$,
\[
\E\left(\sup_{0\le r\le t}\|u_\var(r)-\bar u(r)\|^2\right)\le A_{\var,R,T}
+C_{R,T}\int_0^t(t-s)^{-1/2}
\E\left(\sup_{0\le r\le s}\|u_\var(r)-\bar u(r)\|^2\right)\,ds,
\]
where
\[
A_{\var,R,T}:=C_{T,R}\left[
\var^{1/8}+\omega^f(\var^{-1/2})^2+\omega^\gamma(\var^{-1/2})^2
+\omega^g(\var^{-1/2})+R^{-2}\right].
\]
Lemma~\ref{lem:vgronwall} now implies
\begin{align*}
\E\left(\sup_{0\leq t\leq T}\|u_\var(t)-\bar{u}(t)\|^2\right)
\leq C_{T,R}\left[\var^{\frac18}+\omega^f\left(\var^{-\frac12}\right)^2
+\omega^\gamma\left(\var^{-\frac12}\right)^2+\omega^g\left(\var^{-\frac12}\right)+R^{-2}\right].
\end{align*}
Letting $\var\to0$ and then $R\to\infty$ completes the proof.
\end{proof}

\begin{rem}\label{rem:ui}
(i)
The convergence is uniform for initial distributions in subsets of \(\cal{P}_2(H)\) with a common \(V\)-fourth moment bound. More precisely, if
\[\sup_{\mu\in B}\int_H\|x\|_V^4\,\mu(dx)<\infty,\]
then the convergence in Theorem~\ref{thmT} is uniform for \(\mu\in B\).

(ii)
The conclusion of Theorem~\ref{thmT} holds for any finite interval, that is 
\begin{align*}
\lim_{\var\to0}\E\left(\sup_{t\in[s,s+T]}\|u_\var(t)-\bar{u}(t)\|^2\right)=0,\qquad \forall s\in\R.
\end{align*}

(iii)
For any $\tilde{F}\in\cal{H}(F_0)$, 
\(\tilde{f}\), \(\tilde{g}\), and \(\tilde{\gamma}\) satisfy
\ref{item:H-f}, \ref{item:H-g}, \ref{item:H-a} and \ref{item:H-d} with the same constants and the same
averaged coefficients \(\bar f\), \(\bar g\), and \(\bar\gamma\).
Indeed, the assertions first hold for every translate
\(\sigma_\tau F_0\): the bounds in \ref{item:H-f}, \ref{item:H-g},
and \ref{item:H-d} are uniform in time, while the supremum over
\(a\in\R\) makes \ref{item:H-a} translation invariant. For a general
\(\tilde F\in\cal H(F_0)\), choose
\(\sigma_{\tau_n}F_0\to\tilde F\) in \(d\). Passing to the limit on each
bounded time interval preserves \ref{item:H-f}, \ref{item:H-g}(i),
\ref{item:H-a}, and \ref{item:H-d}. The \(V\)-valued bounds in
\ref{item:H-g}(ii) pass to the limit by weak lower semicontinuity in
\(L_2(U,V)\). Thus every element of the hull has the stated uniform
properties.
\end{rem}

\subsection{Proof of the second Bogolyubov theorem}

Having established the convergence on finite time intervals in Section \ref{secFBT}, 
we now extend the result to the whole real line.

\begin{proof}[Proof of Theorem~\ref{thmR}]
The conclusions (i) and (ii) follow from
Lemma~\ref{lem:bounded-mild-solution}, while
\eqref{entire-V4} follows from Lemma~\ref{lem:entire-V4}.
It remains to prove \eqref{eq0802}. 
Let
\[
F_\var(r,x,\mu):=-\gamma\left(\frac{r}{\var}\right)|x|^2x
+f\left(\frac{r}{\var},x,\mu\right),
\]
\[
\bar F(x,\mu):=-\bar\gamma |x|^2x+\bar f(x,\mu),\quad
g_\var(r,x,\mu):=g\left(\frac{r}{\var},x,\mu\right).
\]
By Lemma~\ref{lem:bounded-mild-solution}, one sees that
\[
M:=\sup_{0<\var\le1}\sup_{t\in\R}
\left(2\E\|u_\var(t)\|^2+2\E\|\bar u(t)\|^2\right)<\infty.
\]

Fix \(N>0\). In view of Lemma~\ref{lem:bounded-mild-solution}, 
we have for every \(t\in\R\),
\begin{align*}
u_\var(t)
&=S(N)u_\var(t-N)
+\int_{t-N}^{t}S(t-r)F_\var(r,u_\var(r),\Law{u_\var(r)})\,dr \notag\\
&\quad
+\int_{t-N}^{t}S(t-r)g_\var(r,u_\var(r),\Law{u_\var(r)})\,dW_r,
\end{align*}
and
\begin{align*}
\bar u(t)=S(N)\bar u(t-N)
+\int_{t-N}^{t}S(t-r)\bar F(\bar u(r),\Law{\bar u(r)})\,dr 
+\int_{t-N}^{t}S(t-r)\bar g(\bar u(r),\Law{\bar u(r)})\,dW_r.
\end{align*}
Let \(\bar{u}(s,t-N,u_{\var}(t-N)), s\in[t-N,t]\) be the mild solution of the averaged
equation starting from \(u_\var(t-N)\) at time \(t-N\), that is, 
\begin{align*}
\bar{u}(s,t-N,u_{\var}(t-N))
&
=S(s-(t-N))u_\var(t-N)\\
&\quad
+\int_{t-N}^{s}S(s-r)\bar F(\bar{u}(r,t-N,u_{\var}(t-N)),\Law{\bar{u}(r,t-N,u_{\var}(t-N))})\,dr\\
&\quad
+\int_{t-N}^{s}S(s-r)\bar g(\bar{u}(r,t-N,u_{\var}(t-N)),\Law{\bar{u}(r,t-N,u_{\var}(t-N))})\,dW_r.
\end{align*}
Then we have
\begin{align}\label{23:3}
&
\E\|u_\var(t)-\bar u(t)\|^2\notag\\
&\le
2\E\|u_\var(t)-\bar{u}(t,t-N,u_{\var}(t-N))\|^2
+2\E\|\bar{u}(t,t-N,u_{\var}(t-N))-\bar u(t)\|^2.
\end{align}
Since \(\bar{u}(\cdot,t-N,u_{\var}(t-N))\) and \(\bar u\) solve the averaged equation on
\([t-N,t]\) with initial values \(u_\var(t-N)\) and \(\bar u(t-N)\),
respectively, Lemma~\ref{Law-con} gives
\begin{align}\label{23:4}
\E\|\bar{u}(t,t-N,u_{\var}(t-N))-\bar u(t)\|^2
&\le e^{-\rho N}\E\|u_\var(t-N)-\bar u(t-N)\|^2 \notag\\
&\le e^{-\rho N}M,
\end{align}
where $\rho:=2(\lambda_1+\lambda)-4L_f-4L_g^2>0$.

For \(a\in\R\), set
\[
W_r^a:=W_{a+r}-W_a,\quad
u_\var^a(r):=\sigma_au_\var(r):=u_\var(a+r),\quad
V_\var^a(r):=\bar{u}(a+r,a,u_\var(a)),\quad 0\le r\le N.
\]
Hence, we obtain
\begin{align*}
u_\var^a(r)
&=S(r)u_\var(a)
+\int_0^rS(r-s)F_\var^a(s,u_\var^a(s),\Law{u_\var^a(s)})\,ds\\
&\quad
+\int_0^rS(r-s)g_\var^a(s,u_\var^a(s),\Law{u_\var^a(s)})\,dW_s^a,
\end{align*}
and
\begin{align*}
V_\var^a(r)
=S(r)u_\var(a)
+\int_0^rS(r-s)\bar F(V_\var^a(s),\Law{V_\var^a(s)})\,ds
+\int_0^rS(r-s)\bar g(V_\var^a(s),\Law{V_\var^a(s)})\,dW_s^a.
\end{align*}

Therefore, Theorem~\ref{thmT},
Remark \ref{rem:ui} (i) and (iii), and 
\eqref{entire-V4} give
\begin{align}\label{eq0713:01}
&
\sup_{a\in\R}
\E\left(\sup_{0\le r\le N}\|u_\var^a(r)-V_\var^a(r)\|^2\right)\notag\\
&\le C_{N,R}
\left[\var^{\frac18}
+\omega^f\left(\var^{-\frac12}\right)^2
+\omega^\gamma\left(\var^{-\frac12}\right)^2
+\omega^g\left(\var^{-\frac12}\right)\right]+C_NR^{-2}.
\end{align}
Taking \(a=t-N\) and \(r=N\), we get
\begin{equation*}
u_\var(t)=u_\var^{t-N}(N),\qquad
\bar{u}(t,t-N,u_{\var}(t-N))=V_\var^{t-N}(N),
\end{equation*}
Together with \eqref{23:3}, \eqref{23:4} and \eqref{eq0713:01}, we have
\begin{align*}
&
\sup_{t\in\R}\E\|u_\var(t)-\bar u(t)\|^2\\
&\le
2\sup_{t\in\R}\E\|u_\var(t)-\bar{u}(t,t-N,u_{\var}(t-N))\|^2
+2\sup_{t\in\R}\E\|\bar{u}(t,t-N,u_{\var}(t-N))-\bar u(t)\|^2 \\
&
\leq C_{N,R}
\left[\var^{\frac18}
+\omega^f\left(\var^{-\frac12}\right)^2
+\omega^\gamma\left(\var^{-\frac12}\right)^2
+\omega^g\left(\var^{-\frac12}\right)\right]+C_NR^{-2}
+e^{-\rho N}M.
\end{align*}
Letting $\var\to0$, then $R\to\infty$, and finally $N\to\infty$ completes the proof.
\end{proof}
  
\section{Attractors in the Space of Probability Measures}
We now study the law-valued dynamics of \eqref{eq:SPDEone} and
\eqref{eq:SPDEtwo}. The fast time scale changes the driving action: for
$0<\var\le1$, define
\begin{equation}\label{scaled-base-flow}
\sigma_t^\var F:=\sigma_{t/\var}F,\qquad t\in\R,\quad F\in\cal H(F_0).
\end{equation}
This scaling is essential for the cocycle identity of the equation with
coefficients $F(t/\var,\cdot)$.

\subsection{Cocycle and dissipativity}

We use a cocycle formulation that permits a noncompact symbol space.
\begin{de}\label{def0617:01}
Let \(\cal X\) be a complete metric space and let \(\cal P\) be a set
equipped with a group action
\(\{\sigma_t\}_{t\in T}\), where \(T=\Z\) or \(\R\). Thus
\(\sigma_0=\operatorname{Id}_{\cal P}\) and
\(\sigma_{t+s}=\sigma_t\sigma_s\).
A map
\(\Phi:T^+\times\cal P\times\cal X\to\cal X\) is called a
\emph{continuous cocycle} over \(\sigma\) if
\begin{itemize}[leftmargin=2em]
\item[(i)] \(\Phi(0,F,x)=x\) for all
\((F,x)\in\cal P\times\cal X\);
\item[(ii)] \(\Phi(t+s,F,x)
=\Phi(t,\sigma_sF,\Phi(s,F,x))\) for all \(s,t\in T^+\) and
\((F,x)\in\cal P\times\cal X\);
\item[(iii)] \(x\mapsto\Phi(t,F,x)\) is continuous for every
\(t\in T^+\) and \(F\in\cal P\).
\end{itemize}
Here \(\cal P\) is the \emph{base space} and \(\cal X\) is the fiber
(or state) space. 
\end{de} 

\begin{de}
A bounded set \(B\subset \cal P_2(H)\) is called a {\em pullback absorbing set} for \(\Phi\) with respect to all bounded subsets of \(\cal P_2(H)\), if for every bounded \(K\subset\cal P_2(H)\) and every \(F\in\cal H(F_0)\), there exists \(T_K>0\) such that
\[\Phi(t,\sigma_{-t}F,K)\subset B,\qquad t\ge T_K.\]
\end{de}

\begin{de}\label{def0617:02}
A family \(\{\cal A(F)\}_{F\in\cal H(F_0)}\) is called a {\em pullback attractor} for \(\Phi\) with respect to all bounded subsets of \(\cal P_2(H)\) if
\begin{itemize}
\item[(i)] \(\cal A(F)\) is compact in \((\cal P_2(H),W_2)\) for every \(F\in\cal H(F_0)\);
\item[(ii)] it is invariant, i.e.
\[\Phi(t,F,\cal A(F))=\cal A(\sigma_tF),\qquad t\ge0;\]
\item[(iii)] it pullback attracts every bounded \(K\subset\cal P_2(H)\):
\[\lim_{t\to\infty}dist\bigl(\Phi(t,\sigma_{-t}F,K),\cal A(F)\bigr)=0.\]
\end{itemize}
\end{de}

Combining the uniqueness of solutions to \eqref{eq:SPDEone} and Lemma~\ref{Law-con}, 
we have the following lemma. 
\begin{lem}\label{cocycle}
Assume that \ref{item:H-f} and \ref{item:H-g} hold. 
Then for each $0<\var\leq 1$,
$$\Phi^\var:\R_+\times\cal{H}(F_0)\times \cal{P}_2(H)\rightarrow \cal{P}_2(H)$$ 
is a cocycle over the scaled translation action $\sigma^\var$ in
\eqref{scaled-base-flow}.

Moreover, if \ref{item:H-a} holds, then $\{\bar P_t^*\}_{t\ge0}$ 
is a semigroup on $\cal P_2(H)$.
\end{lem}

\begin{lem}
\label{abset}
Assume that \((\mathbf H1),(\mathbf H2)\) and \ref{item:H-d} hold. If
$2(\lambda_1+\lambda)>c_1+c_2$, then there exists $R>0$ such that
\[
B:=\left\{\mu\in\cal P_2(H):\int_H\|x\|^2\,\mu(dx)\le R\right\}
\]
is uniformly pullback absorbing, that is, for every bounded
$K\subset\cal P_2(H)$ there exists $T_K>0$ such that
\[
\Phi^\var(t,\sigma_{-t}^\var F,K)\subset B,
\qquad t\ge T_K,\quad F\in\cal H(F_0),\quad 0<\var\le1.
\]
\end{lem}
\begin{proof}
Let \(u_\var(t)\) be the solution of \eqref{eq:SPDEone} with initial law \(\mu_0\in K\). 
Applying It\^o's formula, the Poincar\'e inequality and \ref{item:H-d}, we obtain
\begin{align*}
\frac{d}{dt}\E\norm{u_\var(t)}^2
&
=2\E\inpro{Au_\var(t),u_\var(t)}-2\gamma\left(\frac{t}{\var}\right)\E\|u_\var(t)\|_{L^4}^4 
+2\E\inpro{f\left(\frac{t}{\var},u_\var(t),\Law{u_\var(t)}\right),u_\var(t)}\\
&\qquad
+\E\norm{g\left(\frac{t}{\var},u_\var(t),\Law{u_\var(t)}\right)}_{L_2(U,H)}^2\\
&
\leq \bigl(c_1+c_2-2(\lambda_1+\lambda)\bigr)\E\norm{u_\var(t)}^2+c_0,
\end{align*}
which implies that
\begin{align}\label{r2-estimate}
\E\norm{u_\var(t)}^2\le e^{-\alpha t}\E\norm{u_0}^2+\frac{c_0}{\alpha},
\end{align}
where $\alpha:=2(\lambda_1+\lambda)-c_1-c_2>0$. 
Hence, \eqref{r2-estimate}, Remark~\ref{rem:ui} (iii) 
and the boundedness of $K$ yield that
there exists $M_K>0$ so that
\begin{align*}
\sup_{\substack{\mu\in K,\ F\in\cal H(F_0)\\0<\var\le1}}
\int_H\|x\|^2\,\Phi^\var(t,F,\mu)(dx)
\le \sup_{\mu\in K}e^{-\alpha t}\int_H\|x\|^2\mu(dx)+\frac{c_0}{\alpha}
\leq e^{-\alpha t}M_K+\frac{c_0}{\alpha}.
\end{align*}
Choose $T_K$ large enough that $e^{-\alpha t}M_K < 1$ for all $t\ge T_K$.
Taking $R:=1+c_0/\alpha$ proves the assertion.
\end{proof}

\subsection{Uniform asymptotic compactness}

To verify compactness of the pullback dynamics, we use
the following consequence of Prokhorov's theorem and
\cite[Theorem 6.9]{villani2009optimal}.
\begin{lem}
\label{W2compact}
Let $H$ be a separable Hilbert space, and let $K \subset \cal{P}_2(H)$. 
$K$ is relatively compact in $(\cal{P}_2(H),W_2)$ if and only if
\begin{enumerate}
\item $K$ is tight in $\cal{P}(H)$.
\item \begin{align*}
\lim_{R\to\infty}\sup_{\mu\in K}\int_{\{\|x\|>R\}}\|x\|^2\,\mu(dx)=0.
\end{align*}
\end{enumerate}
Moreover, if $K$ is closed, then $K$ is compact in $(\cal{P}_2(H),W_2)$.
\end{lem}

\begin{lem}\label{lem:uniformac}
Assume that \ref{item:H-f}, \ref{item:H-g} and \ref{item:H-d} hold with
$2(\lambda_1+\lambda)>\max\{c_1+c_2,c_1+2L_g^2\}$.
Let $K\subset\cal P_2(H)$ be bounded. Then for any
$\{\var_n\}\subset(0,1]$, $t_n\to\infty$, $\mu_n\in K$, and
$F_n\in\cal H(F_0)$,
$\Phi^{\var_n}(t_n,\sigma_{-t_n}^{\var_n}F_n,\mu_n)$
admits a convergent subsequence in $(\cal P_2(H),W_2)$.
\end{lem}

\begin{proof}
Since $K$ is bounded in $\cal P_2(H)$, there exists a constant $M_K>0$ such that
\begin{align}\label{48:1}
M_K:=\sup_{\mu\in K}\int_H \|x\|^2\,\mu(dx)<\infty.
\end{align}
For every $n$, choose an $\cal F_0$-measurable random variable $\xi_n$ such that
$\Law{\xi_n}=\mu_n$.
 Set
 \[
 u_n^{\var_n}(t):=u_{\var_n}(t;0,\xi_n,\sigma_{-t_n}^{\var_n}F_n).
 \]
 Then
 \[\Phi^{\var_n}(t_n,\sigma_{-t_n}^{\var_n}F_n,\mu_n)=\Law{u_n^{\var_n}(t_n)}.\]
 In the estimates below, we write
 \(\sigma_{-t_n}^{\var_n}F_n=(\gamma_n,f_n,g_n)\); by
 Remark~\ref{rem:ui} (iii), these coefficients satisfy the assumptions with the
 same constants as \((\gamma,f,g)\).
 By \eqref{r2-estimate} and \eqref{48:1}, we have
 \[\E\norm{u_n^{\var_n}(t)}^2\le e^{-\alpha t}M_K+\frac{c_0}{\alpha}\le M_K+\frac{c_0}{\alpha}=:C_K,\qquad t\ge0,\]
 which implies that
 \begin{align}\label{48:2}
 \sup_{n\ge1}\sup_{t\ge0}\E\norm{u_n^{\var_n}(t)}^2\le C_K.
 \end{align}
 
 We first prove that $\{\Law{u_n^{\var_n}(t_n)}\}_{n\ge1}$ is tight in $\cal P(H)$.
Since $t_n\to\infty$, without loss of generality, we assume that $t_n>1$ for all $n$.
It\^o's formula, \ref{item:H-d} and the Poincar\'e inequality give 
for any $\theta\in(0,\lambda_1^{-1}(2\lambda_1+2\lambda-c_1-c_2))$,
\begin{align*}
 \frac{d}{dt}\E\norm{u_n^{\var_n}(t)}^2
 &
 \le -\theta\E\norm{\nabla u_n^{\var_n}(t)}^2-(2-\theta)\E\norm{\nabla u_n^{\var_n}(t)}^2+(c_1+c_2-2\lambda)\E\norm{u_n^{\var_n}(t)}^2+c_0\\
 &
 \leq -\theta\E\norm{\nabla u_n^{\var_n}(t)}^2
 -\alpha_1\E\norm{u_n^{\var_n}(t)}^2+c_0,
\end{align*}
where $\alpha_1:=(2-\theta)\lambda_1+2\lambda-c_1-c_2>0$.
Consequently,
 \[\E\norm{\nabla u_n^{\var_n}(t)}^2\le \frac{1}{\theta}\left(-\frac{d}{dt}
 \E\norm{u_n^{\var_n}(t)}^2-\alpha_1\E\norm{u_n^{\var_n}(t)}^2+c_0\right).\]
Integrating this inequality over $[t_n-1,t_n]$ and using \eqref{48:2}, we obtain
\begin{align}\label{48:3}
 \E\int_{t_n-1}^{t_n}\norm{\nabla u_n^{\var_n}(s)}^2\,ds
 &
 \le\frac{1}{\theta}\left(\E\norm{u_n^{\var_n}(t_n-1)}^2-\E\norm{u_n^{\var_n}(t_n)}^2
 -\alpha_1\int_{t_n-1}^{t_n}\E\norm{u_n^{\var_n}(s)}^2\,ds+c_0\right) \notag\\
 &
 \le\frac{1}{\theta}\left(\E\norm{u_n^{\var_n}(t_n-1)}^2+c_0\right)\le C.
\end{align}
Then \eqref{48:3} implies that there exists
$s_n\in[t_n-1,t_n]$ such that
 \[\E\norm{\nabla u_n^{\var_n}(s_n)}^2\le C.\]
 Combining this estimate with \eqref{48:2}, we obtain
 \[\sup_n \E\norm{u_n^{\var_n}(s_n)}_{H^1}^2=\sup_n \E\left(\norm{\nabla u_n^{\var_n}(s_n)}^2+\norm{u_n^{\var_n}(s_n)}^2\right)\le C.\]
Moreover, $0\le t_n-s_n\le1$. Applying \eqref{Hnorm} on
$[s_n,t_n]$ and using the flow property therefore gives
\begin{align*}
 \E\|u_n^{\var_n}(t_n)\|_{H^1}^2
 &
  =\E\|u_n^{\var_n}(t_n;s_n,u_{n}^{\var_n}(s_n),\sigma_{-t_n}^{\var_n}F_n)\|_{H^1}^2\\
 &
 \leq C\left(1+\E\|u_n^{\var_n}(s_n)\|_{H^1}^2\right)\leq C,
 \end{align*}
 where $C>0$ is independent of $n$, $\var_n$, and $F_n$, 
because of Remark~\ref{rem:ui} (iii).
Define
\[K_R:=\overline{\{u\in V:\norm{u}_{H^1}\le R\}}.\]
By the compact embedding $H^1\hookrightarrow H$, $K_R$ is a compact set in $H$.
Chebyshev's inequality yields
\[
 \P(u_n^{\var_n}(t_n)\notin K_R)\le\frac{1}{R^2}\E\norm{u_n^{\var_n}(t_n)}_{H^1}^2\le\frac{C}{R^2}.
 \]
 Thus $\{\Law{u_n^{\var_n}(t_n)}\}_{n=1}^\infty$ is tight in $\cal P(H)$.
 
 It remains to verify the uniform integrability of the second moments. Fix $L\ge1$.
 We split the tail of $u_n^{\var_n}(t_n)$ according to $\{\norm{\xi_n}\le L\}$ and $\{\norm{\xi_n}> L\}$.
First consider the part on $\{\norm{\xi_n}\le L\}$. 
For $m>L$, define the stopping time
\[
 \tau_m^{n}:=\inf\{t\ge0:\norm{u_{n}^{\var_n}(t)}\ge m\}.
\]
Let $\beta\in(0,4(\lambda_1+\lambda)-2c_1-4L_g^2)$.
Applying the It\^o formula directly to
$e^{\beta(t\wedge\tau_m^{n})}
 \norm{u_{n}^{\var_n}(t\wedge\tau_m^{n})}^4$ and multiplying by
$\mathbf 1_{\{\norm{\xi_n}\le L\}}$, we obtain
\begin{align}\label{eq0731:01}
&
\mathbf 1_{\{\norm{\xi_n}\le L\}}
e^{\beta(t\wedge\tau_m^{n})}
 \norm{u_n^{\var_n}(t\wedge\tau_m^{n})}^4\notag\\
&
=\mathbf 1_{\{\norm{\xi_n}\le L\}}\norm{\xi_n}^4\notag\\
&\quad
+\int_0^t\mathbf 1_{\{\norm{\xi_n}\le L\}}
\mathbf 1_{\{s<\tau_m^n\}}e^{\beta s}
 \Bigg\{\beta\norm{u_n^{\var_n}(s)}^4
 +4\norm{u_n^{\var_n}(s)}^2\inpro{Au_n^{\var_n}(s),u_n^{\var_n}(s)}\notag\\
 &\hspace{13mm}
 -4\gamma_n\left(\frac{s}{\var_n}\right)\norm{u_n^{\var_n}(s)}^2
 \norm{u_n^{\var_n}(s)}_{L^4}^4\notag\\
 &\hspace{13mm}
 +4\norm{u_n^{\var_n}(s)}^2
 \inpro{f_n\left(\frac{s}{\var_n},u_n^{\var_n}(s),\Law{u_n^{\var_n}(s)}\right),u_n^{\var_n}(s)}\notag\\
 &\hspace{13mm}
 +2\norm{u_n^{\var_n}(s)}^2
 \norm{g_n\left(\frac{s}{\var_n},u_n^{\var_n}(s),\Law{u_n^{\var_n}(s)}\right)}_{L_2(U,H)}^2\notag\\
 &\hspace{13mm}
 +4\norm{g_n\left(\frac{s}{\var_n},u_n^{\var_n}(s),
 \Law{u_n^{\var_n}(s)}\right)^*u_n^{\var_n}(s)}_U^2\Bigg\}\,ds\notag\\
&\quad
+4\int_0^t\mathbf 1_{\{\norm{\xi_n}\le L\}}
 \mathbf 1_{\{s<\tau_m^n\}}e^{\beta s}\norm{u_n^{\var_n}(s)}^2\notag\\
 &\hspace{13mm}\times\inpro{g_n\left(\frac{s}{\var_n},u_n^{\var_n}(s),
 \Law{u_n^{\var_n}(s)}\right)dW_s,u_n^{\var_n}(s)}.
\end{align}

Consequently, using \eqref{eq0731:01}, \eqref{eq0731:02}, $\gamma_n\geq0$, \ref{item:H-d} and \eqref{48:2}, we obtain for any $\iota>0$,
\begin{align*}
&
\E\left[\mathbf 1_{\{\norm{\xi_n}\le L\}}
e^{\beta(t\wedge\tau_m^{n})}
 \norm{u_{n}^{\var_n}(t\wedge\tau_m^{n})}^4\right]\notag\\
&
\leq \E\left(\mathbf 1_{\{\norm{\xi_n}\le L\}}\norm{\xi_n}^4\right)
+\notag\\
&\quad
+\E\int_0^t\mathbf 1_{\{\norm{\xi_n}\le L\}}
\mathbf 1_{\{s<\tau_m^n\}}e^{\beta s}
 \Bigg\{\left[\beta-(4(\lambda_1+\lambda)-2c_1-4L_g^2)+5\iota\right]\norm{u_n^{\var_n}(s)}^4
+C_K\Bigg\}ds,
\end{align*}
which by letting $0<\iota<\frac15\left(4(\lambda_1+\lambda)-2c_1-4L_g^2-\beta\right)$ implies that
\begin{align}\label{48:4}
\E\left[\mathbf 1_{\{\norm{\xi_n}\le L\}}
e^{\beta(t\wedge\tau_m^{n})}
 \norm{u_{n}^{\var_n}(t\wedge\tau_m^{n})}^4\right]
&
\leq L^4
+C_K\E\int_0^t\mathbf 1_{\{\norm{\xi_n}\le L\}}
\mathbf 1_{\{s<\tau_m^n\}}e^{\beta s}
ds\notag\\
&
\leq L^4+\frac{C_K}{\beta}({\rm e}^{\beta t}-1).
\end{align}
With the help of \eqref{48:4}, we obtain
\begin{align*}
\E\left(\mathbf 1_{\{\norm{\xi_n}\le L\}}\mathbf 1_{\{t<\tau_m^{n}\}}
 \norm{u_{n}^{\var_n}(t)}^4\right)
\le e^{-\beta t}L^4+\frac{C_K}{\beta}(1-e^{-\beta t})
\le C_{K,L},\qquad t\ge0.
\end{align*}
Hence, the monotone convergence theorem yields
\[
 \E\left(\mathbf 1_{\{\norm{\xi_n}\le L\}}\norm{u_{n}^{\var_n}(t)}^4\right)\le C_{K,L},\qquad t\ge0.
\]
Therefore, for every $R>0$,
\begin{align}\label{48:8}
 \E\left(\norm{u_n^{\var_n}(t_n)}^2\mathbf 1_{\{\norm{u_n^{\var_n}(t_n)}>R\}}\mathbf 1_{\{\norm{\xi_n}\le L\}}\right)
 \le\frac{1}{R^2}\E\left(\mathbf 1_{\{\norm{\xi_n}\le L\}}\norm{u_n^{\var_n}(t_n)}^4\right) \le\frac{C_{K,L}}{R^2}.
\end{align}

Next consider the part on $\{\norm{\xi_n}> L\}$. 
By \ref{item:H-d}, H\"older's inequality, \eqref{48:2} and Chebshev's inequality, we have
\begin{align*}
&
 \frac{d}{dt}\E\left(\mathbf 1_{\{\norm{\xi_n}> L\}}\norm{u_n^{\var_n}(t)}^2\right)\\
 &
 \le -2(\lambda_1+\lambda)\E\left(\mathbf 1_{\{\norm{\xi_n}> L\}}\norm{u_n^{\var_n}(t)}^2\right) \\
 &\quad
 +\E\left[\mathbf 1_{\{\norm{\xi_n}> L\}}
 \left(2\inpro{f_n\left(\frac{t}{\var_n},u_n^{\var_n}(t),\Law{u_n^{\var_n}(t)}\right),u_n^{\var_n}(t)}
 +\norm{g_n\left(\frac{t}{\var_n},u_n^{\var_n}(t),\Law{u_n^{\var_n}(t)}
 \right)}_{L_2(U,H)}^2\right)\right]\\
 &
 \leq -\left[2(\lambda_1+\lambda)-c_1\right]\E\left(\mathbf 1_{\{\norm{\xi_n}> L\}}\norm{u_n^{\var_n}(t)}^2\right)
 +\P(\{\norm{\xi_n}> L\})\left(c_0+c_2\E\norm{u_n^{\var_n}(t)}^2\right)\\
 &
 \leq -\left[2(\lambda_1+\lambda)-c_1\right] \E\left(\mathbf 1_{\{\norm{\xi_n}> L\}}\norm{u_n^{\var_n}(t)}^2\right)+C_K\P(\{\norm{\xi_n}> L\})\\
 &
 \leq-\left[2(\lambda_1+\lambda)-c_1\right] \E\left(\mathbf 1_{\{\norm{\xi_n}> L\}}\norm{u_n^{\var_n}(t)}^2\right)
+\frac{C_K}{L^2}.
\end{align*}
Hence, 
\begin{align}\label{48:9}
 \E\left(\mathbf 1_{\{\norm{\xi_n}> L\}}\norm{u_n^{\var_n}(t_n)}^2\right)
\leq e^{-\left[2(\lambda_1+\lambda)-c_1\right]t_n}M_K
+\frac{C_K}{L^2}.
\end{align}
Combining \eqref{48:8} and \eqref{48:9}, we obtain
\begin{align}\label{48:10}
 \E\left[\norm{u_n^{\var_n}(t_n)}^2\mathbf 1_{\{\norm{u_n^{\var_n}(t_n)}>R\}}\right]
 &
 \le\E\left[\norm{u_n^{\var_n}(t_n)}^2\mathbf 1_{\{\norm{u_n^{\var_n}(t_n)}>R\}}\mathbf 1_{\{\norm{\xi_n}\le L\}}\right]\notag\\
 &\quad+\E\left[\norm{u_n^{\var_n}(t_n)}^2\mathbf 1_{\{\norm{\xi_n}> L\}}\right] \notag\\
&
\le\frac{C_{K,L}}{R^2}
+e^{-\left[2(\lambda_1+\lambda)-c_1\right] t_n}M_K+\frac{C_K}{L^2}.
\end{align}
Let $\eta>0$ be arbitrary.
Since $t_n\to\infty$ and $2(\lambda_1+\lambda)-c_1>0$, there exists $N\in\mathbb N$ such that
\[
e^{-\left[2(\lambda_1+\lambda)-c_1\right] t_n}M_K<\frac{\eta}{3},\qquad \forall n\ge N.
\]
Then choose $L\ge1$ sufficiently large such that $\frac{C_K}{L^2}<\frac{\eta}{3}$.
For this fixed $L$, there exists $R_1>0$ such
that for all $R>R_1$, $\frac{C_{K,L}}{R^2}<\frac{\eta}{3}$.
Hence, by \eqref{48:10}, for all $n\ge N$
and $R>R_1$,
 \[\E\left(\norm{u_n^{\var_n}(t_n)}^2\mathbf 1_{\{\norm{u_n^{\var_n}(t_n)}>R\}}\right)<\eta.\]
 In the case where $1\le n<N$, since $u_n^{\var_n}(t_n)\in L^2(\Omega;H)$, we have
 \[
 \lim_{R\to\infty}\max_{1\le n<N}\E\left(\norm{u_n^{\var_n}(t_n)}^2\mathbf 1_{\{\norm{u_n^{\var_n}(t_n)}>R\}}\right)=0.
 \]
Therefore, there exists $R_2\geq R_1$ such that
for all $R>R_2$,
 \[\sup_{n\ge1}\E\left(\norm{u_n^{\var_n}(t_n)}^2\mathbf 1_{\{\norm{u_n^{\var_n}(t_n)}>R\}}\right)<2\eta.\]
 Equivalently,
 \[\lim_{R\to\infty}\sup_{n\ge1}\int_{\{\norm{x}>R\}}\norm{x}^2\,\Law{u_n^{\var_n}(t_n)}(dx)=0.\]
 Together with the tightness proved,
 Lemma~\ref{W2compact} implies that
 $\{\Law{u_n^{\var_n}(t_n)}\}$ is relatively compact in $(\cal P_2(H),W_2)$. 
\end{proof}

The standard pullback-attractor construction (see, e.g., \cite{CLR2006}), 
together with Lemmas \ref{cocycle}, \ref{abset},
and \ref{lem:uniformac}, gives the following result.
\begin{prop}\label{attractor-existence}
Assume that \ref{item:H-f}, \ref{item:H-g} and \ref{item:H-d} hold with
$2(\lambda_1+\lambda)>\max\{c_1+c_2,c_1+2L_g^2\}$. 
For every $0<\var\le1$, the cocycle
\(\Phi^\var\) possesses a pullback attractor
$\{\cal A^\var(F)\}_{F\in\cal H(F_0)}$ in $(\cal P_2(H),W_2)$ with respect
to all bounded subsets of \(\cal P_2(H)\). More precisely,
\[\cal{A}^\var(F)=\bigcap_{t\ge0}\overline{\bigcup_{s\ge t}
\Phi^\var(s,\sigma_{-s}^\var F,B)},\]
where \(B\) is defined in Lemma~\ref{abset}. 
\end{prop}

\begin{rem}
For every $0<\var\le1$ and $F\in\cal H(F_0)$, we have
\begin{align}\label{411:0}
\cal A^\var(F)\subset B.
\end{align}
Indeed, applying Lemma~\ref{abset} with $K=B$, there exists $T_B>0$
such that
\[
\Phi^\var(t,\sigma_{-t}^\var F,B)\subset B,
\qquad t\ge T_B,
\]
uniformly with respect to $F\in\cal H(F_0)$ and $0<\var\le1$.
Since $B$ is closed in $(\cal P_2(H),W_2)$, the representation
\[
\cal A^\var(F)
=\bigcap_{t\ge0}\overline{\bigcup_{s\ge t}
\Phi^\var(s,\sigma_{-s}^\var F,B)}
\]
therefore implies \eqref{411:0}.
\end{rem}

\begin{rem}\label{rem0803}
Assume that \ref{item:H-f}--\ref{item:H-d} hold with
$2(\lambda_1+\lambda)>\max\{c_1+c_2,c_1+2L_g^2\}$. 
Using Remarks \ref{rem0618} and \ref{rem0618-2}, 
we obtain that the set \(B\) constructed in Lemma~\ref{abset} is also a 
bounded absorbing set for the averaged semigroup 
\(\{\bar P_t^*\}_{t\ge0}\).
By the same arguments as in
Lemma~\ref{lem:uniformac} and
Proposition~\ref{attractor-existence}, \(\{\bar P_t^*\}_{t\ge0}\) admits a compact global attractor
\[\bar{\cal A}:=\bigcap_{t\ge0}\overline{\bigcup_{s\ge t}\bar P_s^*(B)}.\]
Moreover, \(\bar{\cal A}\) attracts every bounded subset of \(\cal P_2(H)\); in particular,
\begin{equation}\label{eq0803:01}
\lim_{T\rightarrow\infty}dist(\bar P_T^*B,\bar{\cal A})=0.
\end{equation}
\end{rem}

\subsection{Convergence of attractors}

We first record a density result.
\begin{lem}
\label{densityP4V}
The set
\[\cal P_4(V):=\left\{\mu\in\cal P_2(H):\mu(V)=1, \int_H\|x\|_V^4\,\mu(dx)<\infty\right\}\]
is dense in \((\cal P_2(H),W_2)\).
\end{lem}

\begin{proof}
Let \(\mu\in\cal P_2(H)\). For a Borel map \(T:H\to H\), define
\[T_\#\mu(A):=\mu(T^{-1}(A)),\qquad A\in\cal B(H).\]
Let \(\{e_k\}_{k\ge1}\subset V\) be the eigenfunctions of \(-\Delta\), which form an orthonormal basis of \(H\). 
For \(n\in\N\), set
\[H_n:=\operatorname{span}\{e_1,\ldots,e_n\}.\]
As in Lemma~\ref{lemtwo}, let \(P_n:H\to H_n\)
be the orthogonal projection. Then \((P_n)_\#\mu\) is the law obtained by projecting \(\mu\) onto \(H_n\). Since
\(P_nx\to x\) in \(H\) for every \(x\in H\), and
\[\|P_nx-x\|^2\le 4\|x\|^2,\]
by the dominated convergence theorem, we have
\[\int_H\|P_nx-x\|^2\,\mu(dx)\to0,\]
which implies that
\begin{align}\label{49:1}
W_2^2((P_n)_\#\mu,\mu)\le\int_H\|P_nx-x\|^2\,\mu(dx)\to0.
\end{align}
For each \(n\), define the truncation map
\[T_{n,R}x:=
\begin{cases}
P_nx, & \|P_nx\|\le R,\\
R\dfrac{P_nx}{\|P_nx\|}, & \|P_nx\|>R.
\end{cases}
\]
Then \(T_{n,R}x\in H_n\subset V\) and \(\|T_{n,R}x\|\le R\). Since \(H_n\) is finite-dimensional, all norms are equivalent on \(H_n\). 
Therefore, we have
\begin{align*}
\int_V\|y\|_V^4\,(T_{n,R})_\#\mu(dy)
=\int_H\|T_{n,R} x\|_V^4\,\mu(dx)
\leq C_n\int_H\|T_{n,R} x\|^4\,\mu(dx)\leq C_nR^4,
\end{align*}
and hence \((T_{n,R})_\#\mu\in\cal P_4(V)\).
Moreover, for fixed \(n\), we get
\[\lim_{R\rightarrow\infty}\|T_{n,R}x-P_nx\|=0,\]
and
\[\|T_{n,R}x-P_nx\|^2\le 4\|P_nx\|^2\le 4\|x\|^2.\]
Again by dominated convergence, we obtain that for every fixed $n$,
\begin{align}\label{49:2}
\lim_{R\rightarrow\infty} W_2^2((T_{n,R})_\#\mu,(P_n)_\#\mu)
\le\lim_{R\rightarrow\infty}\int_H\|T_{n,R}x-P_nx\|^2\,\mu(dx)=0.
\end{align}
Therefore, choosing \(n\) large and then \(R\) large, by \eqref{49:1} and \eqref{49:2}, we can find measures \((T_{n,R})_\#\mu\in\cal P_4(V)\) arbitrarily close to \(\mu\) in \(W_2\). Thus \(\cal P_4(V)\) is dense in \((\cal P_2(H),W_2)\).
\end{proof}

\begin{lem}\label{lem:finite-time-compact}
Assume that \ref{item:H-f}, \ref{item:H-g} and \ref{item:H-a} hold. Let \(K\subset\cal P_2(H)\) be compact. 
Then for every \(T>0\),
\[\lim_{\var\to0}\sup_{F\in\cal H(F_0)}\sup_{\mu\in K}W_2\left(\Phi^\var(T,F,\mu),\bar P_T^*\mu\right)=0.\]
\end{lem}
\begin{proof}
Fix \(T>0\). Let \(\kappa>0\). By Lemma~\ref{densityP4V}, \(\cal P_4(V)\) is dense in \(\cal P_2(H)\). Since \(K\) is compact in \((\cal P_2(H),W_2)\), 
there exist finitely many measures $\mu_1,\ldots,\mu_N\in\cal P_4(V)$
such that for every \(\mu\in K\), there exists \(i=i(\mu)\in\{1,\ldots,N\}\) satisfying
\begin{align}\label{410:1}
W_2(\mu,\mu_i)<\kappa.
\end{align}
For such \(\mu\) and \(\mu_i\), we have
\begin{align}\label{410:2}
W_2\left(\Phi^\var(T,F,\mu),\bar P_T^*\mu\right)
&\le W_2\left(\Phi^\var(T,F,\mu),\Phi^\var(T,F,\mu_i)\right)
+W_2\left(\Phi^\var(T,F,\mu_i),\bar P_T^*\mu_i\right)\notag\\
&\quad +W_2\left(\bar P_T^*\mu_i,\bar P_T^*\mu\right).
\end{align}
By Lemma~\ref{Law-con} and Remark~\ref{rem:ui} (iii), 
there exists a constant \(C_T>0\) such that for all \(0<\var\le1\),
\[
\sup_{F\in\cal H(F_0)}W_2\left(\Phi^\var(T,F,\mu),\Phi^\var(T,F,\mu_i)\right)
+W_2\left(\bar P_T^*\mu,\bar P_T^*\mu_i\right)
\le C_T W_2(\mu,\mu_i).
\]
Therefore, together with \eqref{410:1} and \eqref{410:2}, we have
\begin{align*}
W_2\left(\Phi^\var(T,F,\mu),\bar P_T^*\mu\right)\le 2C_T\kappa+W_2\left(\Phi^\var(T,F,\mu_i),\bar P_T^*\mu_i\right),
\end{align*}
which implies that
\[\sup_{F\in\cal H(F_0)}\sup_{\mu\in K} W_2\left(\Phi^\var(T,F,\mu),\bar P_T^*\mu\right)\le
2C_T\kappa+\sum_{i=1}^N\sup_{F\in\cal H(F_0)}W_2\left(\Phi^\var(T,F,\mu_i),\bar P_T^*\mu_i\right).\]
Consequently, by Theorem~\ref{thmT} and Remark~\ref{rem:ui} (iii), we have
\[\limsup_{\var\to0}\sup_{F\in\cal H(F_0)}\sup_{\mu\in K}W_2\left(\Phi^\var(T,F,\mu),\bar P_T^*\mu\right)\le 2C_T\kappa.\]
Since \(\kappa>0\) is arbitrary, the desired result follows.
\end{proof}

Now we give the proof of Theorem~\ref{thm:attractor-convergence}.
\begin{proof}[Proof of Theorem~\ref{thm:attractor-convergence}]
Conclusions (1) and (2) follow from Proposition \ref{attractor-existence} 
and Remark \ref{rem0803}.
We prove conclusion (3) by contradiction.
Suppose that the conclusion is false. Then there exist $\eta_0>0$, a sequence $\{\var_n\}$ with $\var_n\to0$, and a sequence $\{F_n\}\subset\cal{H}(F_0)$ such that
\begin{align*}
dist\bigl(\cal{A}^{\var_n}(F_n),\bar{\cal{A}}\bigr)\ge\eta_0,\qquad \forall n\in\N.
\end{align*}
Since \(\cal A^{\var_n}(F_n)\) is compact and the map $\mu\mapsto dist(\mu,\bar{\cal A})$
is continuous, for each \(n\) one can choose $\mu_n\in\cal A^{\var_n}(F_n)$
such that
\begin{equation}\label{24:1}
dist(\mu_n,\bar{\cal A})=dist(\cal A^{\var_n}(F_n),\bar{\cal A})\ge\eta_0.
\end{equation}
Using \eqref{eq0803:01}, we choose $T>0$ so that
\begin{equation}\label{eq0803:02}
dist(\bar P_T^*B,\bar{\cal A})<\frac{\eta_0}{2}.
\end{equation}
By the invariance of $\cal{A}^{\var_n}$, there exists
$\nu_n\in \cal{A}^{\var_n}(\sigma_{-T}^{\var_n}F_n)$
such that $\mu_n=\Phi^{\var_n}(T,\sigma_{-T}^{\var_n}F_n,\nu_n)$.
By \eqref{411:0}, we have $\nu_n\in B,~n\in\mathbb N$.
We next show that, up to a subsequence, \(\{\nu_n\}\) is relatively compact in \((\cal P_2(H),W_2)\). For each \(n\), by the invariance of \(\cal A^{\var_n}\), there exists
$\eta_n\in \cal A^{\var_n}(\sigma_{-(T+n)}^{\var_n}F_n)$
such that
$
\nu_n=\Phi^{\var_n}(n,\sigma_{-(T+n)}^{\var_n}F_n,\eta_n).
$
Again, in view of \eqref{411:0}, we have
$\eta_n\in B$.
Hence \(\nu_n\) is a sequence of pullback images starting from the fixed set \(B\subset\cal{P}_2(H)\), with time \(n\to\infty\). 
Set
\[\widetilde F_n:=\sigma_{-T}^{\var_n}F_n\in\cal H(F_0).\]
Then $\sigma_{-n}^{\var_n}\widetilde F_n=\sigma_{-(T+n)}^{\var_n}F_n$,
and hence $\nu_n=\Phi^{\var_n}(n,\sigma_{-n}^{\var_n}\widetilde F_n,\eta_n)$.
By Lemma~\ref{lem:uniformac}, after passing to a
subsequence, still denoted by \(\{\nu_n\}\), there exists
\(\nu\in\cal P_2(H)\) such that
\[
\lim_{n\rightarrow\infty} W_2\left(\nu_n,\nu\right)=0.
\]
Consequently, $K:=\{\nu\}\cup\{\nu_n:n\in\N\}$
is compact in \((\cal P_2(H),W_2)\).

By Lemma~\ref{lem:finite-time-compact}, for this compact set \(K\), we have
\[\lim_{\var\to0}\sup_{F\in\cal H(F_0)}\sup_{\mu\in K}W_2\left(\Phi^\var(T,F,\mu),\bar P_T^*\mu\right)=0.\]
In particular,
\begin{align}\label{24:2}
\lim_{n\rightarrow\infty}W_2(\mu_n,\bar P_T^*\nu_n)
=\lim_{n\rightarrow\infty}
W_2\left(\Phi^{\var_n}(T,\sigma_{-T}^{\var_n}F_n,\nu_n),\bar P_T^*\nu_n\right)=0.
\end{align}
Since $\{\nu_n\}\subset B$, by \eqref{eq0803:02} we have for every $n$,
\begin{align}\label{24:3}
dist(\bar P_T^*\nu_n,\bar{\cal{A}})<\frac{\eta_0}{2}.
\end{align}
By \eqref{24:2}, for all sufficiently large \(n\),
\[W_2(\mu_n,\bar P_T^*\nu_n)<\frac{\eta_0}{2}.\]
Therefore, together with \eqref{24:3}, we obtain
\[dist(\mu_n,\bar{\cal A})\le W_2(\mu_n,\bar P_T^*\nu_n)+dist(\bar P_T^*\nu_n,\bar{\cal A})<\eta_0,\]
which contradicts the choice of \(\mu_n\) in \eqref{24:1}. This completes the proof.
\end{proof}

\appendix
\section{}

\begin{de}\label{defAP}
Let $(\cal X,\rho)$ be a metric space. A function
$\varphi\in C(\R,\cal X)$ is called \emph{almost periodic} if, for
every $\epsilon>0$, the set of its $\epsilon$-almost periods
\[
\mathcal T(\varphi,\epsilon)
:=\left\{\tau\in\R:
\sup_{t\in\R}\rho\bigl(\varphi(t+\tau),\varphi(t)\bigr)<\epsilon
\right\}
\]
is relatively dense in $\R$. Equivalently, for every $\epsilon>0$, there
exists a constant $l=l(\epsilon)>0$ such that
\[
\mathcal T(\varphi,\epsilon)\cap[a,a+l]\ne\varnothing,
\qquad a\in\R.
\]
\end{de}

\begin{lem}\label{lem:mild-solution}
Assume that \ref{item:H-f}, \ref{item:H-g} and \ref{item:H-a} hold.
For any initial data $u_0,\bar u_0\in L^2(\Omega,\cal{F}_0,\P;H)$,
there exist unique $\{\cal{F}_t\}$-adapted processes
\[u_\var(t),\bar{u}(t) \in L^2\big(\Omega;C([0,T];H)\big)\cap L^2\big(\Omega;L^2(0,T;V)\big),\qquad \forall\,T>0,\]
which are the unique mild solutions of \eqref{eq:SPDEone} and \eqref{eq:SPDEtwo}, respectively, and satisfy the integral equations
\begin{align*}
  u_\var(t)&=S(t)u_0+\int_{0}^{t}S(t-s)\left[-\gamma\left(\frac{s}{\var}\right)|u_\var(s)|^2u_\var(s)
  +f\left(\frac{s}{\var}, u_\var(s), \Law{u_\var(s)}\right)\right]\,ds  \\
  &+\int_{0}^{t}{S(t-s)g\left(\frac{s}{\var}, u_\var(s), \Law{u_\var(s)}\right)\,dW_s},
  \quad t\in[0,T],
\end{align*}
and
\begin{align*}
 \bar{u}(t)&=S(t)\bar{u}_0+\int_{0}^{t}S(t-s)[-\bar{\gamma}|\bar{u}(s)|^2\bar{u}(s)
  +\bar{f}\left(\bar{u}(s), \Law{\bar{u}(s)}\right)]\,ds \\
  &+\int_{0}^{t}{S(t-s)\bar{g}\left(\bar{u}(s), \Law{\bar{u}(s)}\right)\,dW_s},
  \quad t\in[0,T],
\end{align*}
where $S(t)=e^{At}$ is the semigroup generated by $A$ on $H$.
\end{lem}
\begin{proof}
Note that standard Galerkin approximation, the
monotonicity of $u\mapsto |u|^2u$, and \eqref{eq0731:02-1}--\eqref{eq0731:02} yield a
solution with the required energy bounds; see
\cite[Chapter 7]{DaPratoZabczyk1992} and
\cite[Chapter 4]{liu2015stochastic}. The Lipschitz continuity in the law
variable and It\^o's formula applied to the difference of two solutions,
together with a local contraction and iteration argument, yield existence
on $[0,T]$, pathwise uniqueness, and continuous dependence.
\end{proof}

\begin{lem}\label{lem:bounded-mild-solution}
Assume that \ref{item:H-f}, \ref{item:H-g} and \ref{item:H-a} hold. If
$\lambda_1+\lambda>2L_f+2L_g^2$, then equations
\eqref{eq:SPDEone} and \eqref{eq:SPDEtwo} admit unique bounded entire mild
solutions $u_\var$ and $\bar u$, respectively, in the sense that
\[\sup_{0<\var\le1}\sup_{t\in\R}\E\norm{u_\var(t)}^2<\infty,\quad 
\sup_{t\in\R}\E\norm{\bar{u}(t)}^2<\infty.\]
Moreover, for every $t\in\mathbb{R}$, they satisfy
\begin{align*}
u_\var(t)&=\int_{-\infty}^{t}S(t-s)[-\gamma\left(\frac{s}{\var}\right)|u_\var (s)|^2u_\var(s)
+f\left(\frac{s}{\var}, u_\var(s), \Law{u_\var(s)}\right)]\,ds \\
&+\int_{-\infty}^{t}{S(t-s)g\left(\frac{s}{\var}, u_\var(s), \Law{u_\var(s)}\right)\,dW_s}.
\end{align*}
\begin{align*}
\bar{u}(t)=\int_{-\infty}^{t}S(t-s)[-\bar{\gamma}|\bar{u}(s)|^2\bar{u}(s)
+\bar{f}\left(\bar{u}(s), \Law{\bar{u}(s)}\right)]\,ds 
+ \int_{-\infty}^{t}{S(t-s)\bar{g}\left(\bar{u}(s), \Law{\bar{u}(s)}\right)\,dW_s}.
\end{align*}
The integrals from $-\infty$ are understood as limits in
$L^2(\Omega;H)$ of the corresponding finite-interval mild formulas.
\end{lem}
\begin{proof}
Applying It\^o's formula, we obtain
\begin{align*}
\frac{d}{dr}\E\|u_\var(r;s,\xi)\|^2
&=2\E\langle Au_\var(r;s,\xi),u_\var(r;s,\xi)\rangle
-2\gamma\left(\frac r\var\right)
\E\|u_\var(r;s,\xi)\|_{L^4}^4\\
&\quad+2\E\left\langle
f\left(\frac r\var,u_\var(r;s,\xi),\Law{u_\var(r;s,\xi)}\right),u_\var(r;s,\xi)
\right\rangle\\
&\quad+\E\left\|
g\left(\frac r\var,u_\var(r;s,\xi),\Law{u_\var(r;s,\xi)}\right)
\right\|_{L_2(U,H)}^2.
\end{align*}
Note that by \ref{item:H-f}, the Cauchy--Schwarz inequality, and
Young's inequality, we have for every \(\iota>0\),
\begin{align}\label{eq0803:03}
&
2\E\left\langle
f\left(\frac r\var,u_\var(r),\Law{u_\var(r)}\right),u_\var(r)
\right\rangle\notag\\
&\le 2L_f\E\|u_\var(r)\|^2
+2L_f\bigl(\E\|u_\var(r)\|^2\bigr)^{1/2}\E\|u_\var(r)\|+2K_f\E\|u_\var(r)\|\notag\\
&\le (4L_f+\iota)\E\|u_\var(r)\|^2+C_\iota.
\end{align}
Since
\[
\langle Au_\var(r),u_\var(r)\rangle
=-\|\nabla u_\var(r)\|^2-\lambda\|u_\var(r)\|^2
\le-(\lambda_1+\lambda)\|u_\var(r)\|^2
\]
and \(\gamma\ge0\), combining \eqref{eq0803:03} and \eqref{eq0803:05}
we obtain
\[
\frac{d}{dr}\E\|u_\var(r)\|^2
\le-\bigl(\rho-2\iota\bigr)\E\|u_\var(r)\|^2+C_\iota,
\]
where $\rho:=2(\lambda_1+\lambda)-4L_f-4L_g^2>0$.
Choosing \(\iota\leq \rho/4\), we obtain, for \(t\ge s\),
\begin{equation}\label{appendix:pullback-H2}
\E\|u_\var(t;s,\xi)\|^2
\le e^{-\frac{\rho}{2}(t-s)}\E\|\xi\|^2+C.
\end{equation}

The proof of Lemma~\ref{Law-con}, applied on an arbitrary interval
$[s,t]$, gives
\begin{equation}\label{entire-contraction}
\E\|u_\var(t;s,\xi)-u_\var(t;s,\eta)\|^2
\le e^{-\rho(t-s)}\E\|\xi-\eta\|^2.
\end{equation}
Hence, for $m>n$ and $t\ge-n$, by
\eqref{entire-contraction} and \eqref{appendix:pullback-H2},
\[
\E\|u_\var(t;-m,0)-u_\var(t;-n,0)\|^2
\leq e^{-\rho(t+n)}\E\|u_\var(-n;-m,0)\|^2
\le C e^{-\rho(t+n)}.
\]
Thus $u_\var(t;-n,0)$ converges in $L^2(\Omega;H)$, locally uniformly in
$t$, to a process \(u_\var(t)\). Estimate
\eqref{appendix:pullback-H2} shows that this process is bounded in
\(L^2(\Omega;H)\), uniformly in \(t\) and \(0<\var\le1\).
For every \(s<t\), the flow property gives
\[
u_\var(t;-n,0)
=u_\var\bigl(t;s,u_\var(s;-n,0)\bigr)
\qquad\text{whenever }n>-s.
\]
Passing to the limit by \eqref{entire-contraction} shows that
\(u_\var(t)=u_\var(t;s,u_\var(s))\); hence \(u_\var\) is an entire mild
solution. Moreover,
\[
\E\|S(t-s)u_\var(s)\|^2
\le e^{-2(\lambda_1+\lambda)(t-s)}
\sup_{r\in\R}\E\|u_\var(r)\|^2\longrightarrow0
\quad\text{as }s\to-\infty.
\]
Letting \(s\to-\infty\) in the finite-interval mild formula gives the
displayed improper formula in the stated \(L^2(\Omega;H)\) sense.
If $u_\var$ and $v_\var$ are two bounded entire solutions, then
\eqref{entire-contraction}, followed by $s\to-\infty$, gives
$\E\|u_\var(t)-v_\var(t)\|^2=0$. This proves uniqueness. All constants
are uniform for $0<\var\le1$, and the averaged equation is handled in
the same way.
\end{proof}

We now give the proof of Lemma~\ref{lemtwo}. 
\begin{proof}[Proof of Lemma~\ref{lemtwo}]
There exist eigenvalue-eigenfunction pairs
\[-\Delta e_k=\alpha_k e_k,\qquad \alpha_k\ge0,\quad k\ge1,\]
such that \(\{e_k\}_{k\ge1}\subset H^2(\T^d;\R)\) forms an orthonormal basis of \(H\). 
Recall that $H_n:=\mathrm{span}\{e_1,\dots,e_n\}$,
and \(P_n:H\to H_n\) is the orthogonal projection. 
Let $\{\tilde{e}_i\}_{i\geq1}$ be an orthonormal basis of $U$.
Fix \(n\ge1\). We consider the following finite-dimensional equation
\begin{align*}
\begin{cases}
du_n(t)&=\Bigl[A u_n(t)-\gamma\left(\frac{t}{\var}\right)P_n|u_n(t)|^2u_n(t)
+P_n f\left(\frac{t}{\var},u_n(t),\Law{u_n(t)}\right)\Bigr]dt\\
&\quad+P_n g\left(\frac{t}{\var},u_n(t),\Law{u_n(t)}\right)\,dW^n_t,\\[1mm]
u_n(0)&=P_nu_0.
\end{cases}
\end{align*}
where $W^n(t):=\sum_{i=1}^n\langle W(t),\tilde{e}_i\rangle_U\tilde{e}_i$.
Then
\[\|u_n(t)\|_V^2=\inpro{-Au_n(t),u_n(t)}.\]

Applying It\^o's formula to \(\|u_n(t)\|_V^{2p}\), we obtain
\begin{align}\label{32:0}
d\|u_n(t)\|_V^{2p}
&
=p\|u_n(t)\|_V^{2p-2}\Bigl[
2\inpro{Au_n(t),u_n(t)}_V
-2\gamma\left(\frac{t}{\var}\right)\inpro{|u_n(t)|^2u_n(t),u_n(t)}_V\notag\\
&\quad 
+2\inpro{P_nf\left(\frac{t}{\var},u_n(t),\Law{u_n(t)}\right),u_n(t)}_V
+\norm{P_ng\left(\frac{t}{\var},u_n(t),\Law{u_n(t)}\right)}_{L_2(U,V)}^2
\Bigr]dt\notag\\
&\quad
+2p(p-1)\|u_n(t)\|_V^{2p-4}
\norm{(P_ng(\frac{t}{\var},u_n(t),\Law{u_n(t)}))_V^{*}u_n(t)}_{U}^{2}\,dt+ dM_t,
\end{align}
where
\[
M_t=2p\int_0^t\|u_n(s)\|_V^{2p-2}
\inpro{P_ng\left(\frac{s}{\var},u_n(s),\Law{u_n(s)}\right)dW^n_s,u_n(s)}_V
\]
is a local martingale and $(P_ng)_V^{*}:V\rightarrow U$ denotes the adjoint of
$P_ng:U\rightarrow V$ with respect to the inner product of $V$.
Note that
\begin{align}\label{32:1}
\inpro{Au_n(t),u_n(t)}_V=\inpro{-A(Au_n(t)), u_n(t)}=-\inpro{Au_n(t),Au_n(t)}=-\norm{Au_n(t)}^2\le 0.
\end{align}
Moreover, since the solution is real-valued,
\[\inpro{|u_n(t)|^2u_n(t),u_n(t)}_V=\int_{\T^d}\nabla(|u_n(t)|^2u_n(t))\cdot \nabla u_n(t)\,dx+\lambda\int_{\T^d}|u_n(t)|^4\,dx\ge 0,\]
we have
\begin{align}\label{32:2}
-\gamma\left(\frac{t}{\var}\right)\inpro{|u_n(t)|^2u_n(t),u_n(t)}_V\le 0.
\end{align}

Since \(f\left(\frac{t}{\var},u_n(t),\Law{u_n(t)}\right)\in H\), 
\(u_n(t)\in H_n\), \(Au_n(t)\in H_n\), and \(P_n\) is the orthogonal projection onto \(H_n\),
we have
\begin{align*}
2\inpro{P_nf\left(\frac{t}{\var},u_n(t),\Law{u_n(t)}\right),u_n(t)}_V
&
=2\inpro{P_nf\left(\frac{t}{\var},u_n(t),\Law{u_n(t)}\right),-Au_n(t)}\\
&
=2\inpro{f\left(\frac{t}{\var},u_n(t),\Law{u_n(t)}\right),-Au_n(t)},
\end{align*}
Young's inequality and \ref{item:H-f} then give
for any $\alpha>0$,
\begin{align}\label{32:3}
2\inpro{f\left(\frac{t}{\var},u_n(t),\Law{u_n(t)}\right),-Au_n(t)}
&\le \alpha\|Au_n(t)\|^2+C_\alpha\|f\left(\frac{t}{\var},u_n(t),\Law{u_n(t)}\right)\|^2\notag\\
&\le \alpha\|Au_n(t)\|^2+C\bigl(1+\|u_n(t)\|^2+\E\|u_n(t)\|^2\bigr)\notag\\
&\le \alpha\|Au_n(t)\|^2+C\bigl(1+\|u_n(t)\|_V^2+\E\|u_n(t)\|_V^2\bigr).
\end{align}
Similarly, by \eqref{eq0731:02} and \ref{item:H-g} (ii), we have
\begin{align}\label{32:4}
\|P_ng(\frac{t}{\var},u_n(t),\Law{u_n(t)})\|_{L_2(U,V)}^2\le C\bigl(1+\|u_n(t)\|_V^2+\E\|u_n(t)\|_V^2\bigr),
\end{align}
and
\begin{align}\label{32:5}
\bigl\|(P_ng(\frac{t}{\var},u_n(t),\Law{u_n(t)}))_V^{*}u_n(t)\bigr\|_{U}^{2}\le C\bigl(1+\|u_n(t)\|_V^2+\E\|u_n(t)\|_V^2\bigr)\|u_n(t)\|_V^2.
\end{align}
Choosing \(\alpha>0\) sufficiently small in \eqref{32:3}, 
inserting \eqref{32:1}, \eqref{32:2}, \eqref{32:3}, \eqref{32:4} and \eqref{32:5} into \eqref{32:0}, we obtain
\[d\|u_n(t)\|_V^{2p}\le C\bigl(1+\|u_n(t)\|_V^2+\E\|u_n(t)\|_V^2\bigr)\|u_n(t)\|_V^{2p-2}\,dt+dM_t.\]
Using Young's inequality and Jensen's inequality,
we deduce that
\[\E\|u_n(t)\|_V^{2p}\le\E\|P_nu_0\|_V^{2p}+C_{p,T}\int_0^t\bigl(1+\E\|u_n(s)\|_V^{2p}\bigr)\,ds.\]
Hence, by Gronwall's inequality, we have
\begin{equation}\label{32:6}
\sup_{0\le t\le T}\E\|u_n(t)\|_V^{2p}\le C_T\left(1+\E\|P_nu_0\|_V^{2p}\right)
\leq C_T\left(1+\E\|u_0\|_V^{2p}\right).
\end{equation}

For $R>0$, define
\[
\tau_R^n:=\inf\{t\ge0:\|u_n(t)\|_V\ge R\}.
\]
By \eqref{32:0}, one sees that
\begin{align}\label{32:7}
&
\E\left(\sup_{0\le r\le t}\|u_n(r\wedge\tau_R^n)\|_V^{2p}\right)\notag\\
&
\quad\le \E\|P_nu_0\|_V^{2p}
+C\E\int_0^{t\wedge\tau_R^n}
\bigl(1+\|u_n(s)\|_V^2+\E\|u_n(s)\|_V^2\bigr)
\|u_n(s)\|_V^{2p-2}\,ds\notag\\
&
\qquad+\E\left(\sup_{0\le r\le t}|M_{r\wedge\tau_R^n}|\right).
\end{align}
Using the Burkholder-Davis-Gundy inequality, \eqref{32:5}, Young's inequality
and \eqref{32:6}, 
we have
\begin{align*}
&
\E\left(\sup_{0\le r\le t}|M_{r\wedge\tau_R^n}|\right)\\
&
\le C_p\E\left[
\int_0^{t\wedge\tau_R^n}\|u_n(s)\|_V^{4p-4}
\bigl\|(P_ng(\tfrac{s}{\var},u_n(s),\Law{u_n(s)}))_V^*u_n(s)\bigr\|_U^2\,ds
\right]^{\frac12}\\
&
\le C_p\E\left[
\sup_{0\le r\le t}\|u_n(r\wedge\tau_R^n)\|_V^p
\left(
\int_0^{t\wedge\tau_R^n}
\bigl(1+\|u_n(s)\|_V^2+\E\|u_n(s)\|_V^2\bigr)
\|u_n(s)\|_V^{2p-2}\,ds
\right)^{\frac12}
\right]\\
&
\le \frac12\E\left(\sup_{0\le r\le t}
\|u_n(r\wedge\tau_R^n)\|_V^{2p}\right)
+C_p\E\int_0^{t\wedge\tau_R^n}
\bigl(1+\|u_n(s)\|_V^2+\E\|u_n(s)\|_V^2\bigr)
\|u_n(s)\|_V^{2p-2}\,ds.
\end{align*}
Combining this estimate with \eqref{32:7} gives
\begin{align}\label{eq0720:01}
&
\E\left(\sup_{0\le r\le t}\|u_n(r\wedge\tau_R^n)\|_V^{2p}\right)\notag\\
&
\leq \E\|P_nu_0\|_V^{2p}
+C\E\int_0^{t\wedge\tau_R^n}
\bigl(1+\|u_n(s)\|_V^2+\E\|u_n(s)\|_V^2\bigr)
\|u_n(s)\|_V^{2p-2}\,ds\notag\\
&\quad
+\frac12\E\left(\sup_{0\le r\le t}
\|u_n(r\wedge\tau_R^n)\|_V^{2p}\right).
\end{align}
Using again Young's inequality, Jensen's inequality, and
\eqref{32:6}, we have
\begin{equation}\label{eq0720:02}
\E\int_0^{t\wedge\tau_R^n}
\bigl(1+\|u_n(s)\|_V^2+\E\|u_n(s)\|_V^2\bigr)
\|u_n(s)\|_V^{2p-2}\,ds
\le C_T\bigl(1+\E\|u_0\|_V^{2p}\bigr).
\end{equation}
Equations~\eqref{eq0720:01} and \eqref{eq0720:02} yield
\[
\E\left(\sup_{0\le t\le T}\|u_n(t\wedge\tau_R^n)\|_V^{2p}\right)
\le C_T\bigl(1+\E\|u_0\|_V^{2p}\bigr),
\]
where the constant is independent of $n$ and $R$. Letting $R\to\infty$ and
using Fatou's lemma yields
\begin{align}\label{32:8}
\E\left(\sup_{0\le t\le T}\|u_n(t)\|_V^{2p}\right)
\le C_T\bigl(1+\E\|u_0\|_V^{2p}\bigr).
\end{align}

Finally, letting $n\to\infty$, the standard compactness argument for Galerkin
approximations and the weak lower semicontinuity of the norm, together with
\eqref{32:8}, imply that
\[
\E\left(\sup_{0\le t\le T}\|u_\var(t)\|_V^{2p}\right)
\le C_T\bigl(1+\E\|u_0\|_V^{2p}\bigr).
\]

The proof for \(\bar u\) is the same.
\end{proof}

\begin{lem}\label{lem:entire-V4}
Assume that \ref{item:H-f}, \ref{item:H-g} and \ref{item:H-a} hold and $\lambda_1+\lambda>2L_f+2L_g^2$.
Then the bounded entire solutions obtained in
Lemma~\ref{lem:bounded-mild-solution} satisfy the uniform regularity
estimate 
\begin{equation*}
\sup_{0<\var\le1}\sup_{t\in\R}\E\|u_\var(t)\|_V^4
+\sup_{t\in\R}\E\|\bar u(t)\|_V^4<\infty.
\end{equation*}
\end{lem}

\begin{proof}
We prove the assertion for \(u_\var\). For \(n\in\mathbb N\), consider
the pullback solution \(u_\var(t;-n,0)\), \(t\ge -n\), used in the
proof of Lemma~\ref{lem:bounded-mild-solution}. 
For simplicity, define $u_{\var}^n(t):=u_\var(t;-n,0)$.
It follows from \eqref{appendix:pullback-H2} that
\begin{equation}\label{appendix:H2-entire}
\sup_{0<\var\le1}\sup_{n\in\mathbb N}\sup_{t\ge-n}
\E\|u_{\var}^n(t)\|^2\le C.
\end{equation}
Applying It\^o's formula to
\(\|u_{\var}^n(t)\|^4\), 
by the Poincar\'e inequality, \eqref{eq0731:02-1}, \eqref{eq0731:02},
and \eqref{appendix:H2-entire}, we obtain for every sufficiently small
\(\eta>0\),
\begin{align*}
\frac{d}{dt}\E\|u_{\var}^n(t)\|^4
&
=4\E\left[\|u_{\var}^n(t)\|^2
\langle Au_{\var}^n(t),u_{\var}^n(t)\rangle\right]-4\gamma\left(\frac{t}{\var}\right)
\E\left[\|u_{\var}^n(t)\|^2
\|u_{\var}^n(t)\|_{L^4}^4\right]\\
&\quad+4\E\left[\|u_{\var}^n(t)\|^2
\left\langle
f\left(\frac{t}{\var},u_{\var}^n(t),
\Law{u_{\var}^n(t)}\right),u_{\var}^n(t)
\right\rangle\right]\\
&\quad+2\E\left[\|u_{\var}^n(t)\|^2
\left\|
g\left(\frac{t}{\var},u_{\var}^n(t),
\Law{u_{\var}^n(t)}\right)
\right\|_{L_2(U,H)}^2\right]\\
&\quad+4\E\left\|
g\left(\frac{t}{\var},u_{\var}^n(t),
\Law{u_{\var}^n(t)}\right)^*
u_{\var}^n(t)\right\|_U^2.\\
&\le-\bigl(4(\lambda_1+\lambda)-4L_f-6L_g^2-\eta\bigr)
\E\|u_{\var}^n(t)\|^4+C_\eta.
\end{align*}
Since the pullback solutions start
from zero, the comparison principle yields
\begin{equation}\label{appendix:H4-entire}
\sup_{0<\var\le1}\sup_{n\in\mathbb N}\sup_{t\ge-n}
\E\|u_{\var}^n(t)\|^4\le C.
\end{equation}

Applying It\^o's formula to
\(\|u_{\var}^n(s)\|^2\) and taking expectations, we obtain
\begin{align}\label{eq0803:06}
&
\frac{d}{ds}\E\|u_{\var}^n(s)\|^2
+2\E\|u_{\var}^n(s)\|_V^2
+2\gamma\left(\frac{s}{\var}\right)
\E\|u_{\var}^n(s)\|_{L^4}^4\notag\\
&=2\E\left\langle
f\left(\frac{s}{\var},u_{\var}^n(s),
\Law{u_{\var}^n(s)}\right),u_{\var}^n(s)
\right\rangle+\E\left\|
g\left(\frac{s}{\var},u_{\var}^n(s),
\Law{u_{\var}^n(s)}\right)
\right\|_{L_2(U,H)}^2,
\end{align}
where we have used
\[
\langle Au_{\var}^n(s),u_{\var}^n(s)\rangle
=-\|u_{\var}^n(s)\|_V^2.
\]
Therefore, by \eqref{eq0731:02-1},
\eqref{appendix:H2-entire}, the Cauchy--Schwarz inequality, and
Young's inequality,
\begin{align}\label{eq0803:07}
2\E\left\langle
f\left(\frac{s}{\var},u_{\var}^n(s),
\Law{u_{\var}^n(s)}\right),u_{\var}^n(s)
\right\rangle\le
C\left(1+\E\|u_{\var}^n(s)\|^2\right)\le C.
\end{align}
Similarly, \eqref{eq0731:02} and
\eqref{appendix:H2-entire} imply
\begin{equation}\label{eq0803:08}
\E\left\|
g\left(\frac{s}{\var},u_{\var}^n(s),
\Law{u_{\var}^n(s)}\right)
\right\|_{L_2(U,H)}^2
\le C\left(1+\E\|u_{\var}^n(s)\|^2\right)\le C.
\end{equation}
Since \(\gamma\ge0\), employing \eqref{eq0803:06}, \eqref{eq0803:07} and \eqref{eq0803:08} 
we obtain that
\[
\frac{d}{ds}\E\|u_{\var}^n(s)\|^2
+2\E\|u_{\var}^n(s)\|_V^2\le C.
\]
Integrating over \([r,r+1]\) and using
\eqref{appendix:H2-entire}, we have
\begin{align*}
2\int_r^{r+1}\E\|u_{\var}^n(s)\|_V^2\,ds
\le \E\|u_\var(r;-n,0)\|^2
-\E\|u_\var(r+1;-n,0)\|^2+C\le C.
\end{align*}
We mention that all constants are independent of \(0<\var\le1\), \(n\in\mathbb N\),
and \(r\ge-n\). Hence,
\begin{equation}\label{appendix:V2-average}
\sup_{0<\var\le1}\sup_{n\in\mathbb N}\sup_{r\ge-n}
\int_r^{r+1}\E\|u_{\var}^n(s)\|_V^2\,ds\le C.
\end{equation}
Moreover, It\^o's formula for \(\|u_\var\|_V^2\), together with
\eqref{eq0731:02-1} and \((\mathbf H2)(ii)\), gives
\[
\frac{d}{dt}\E\|u_{\var}^n(t)\|_V^2
\le-\E\|Au_{\var}^n(t)\|^2
+C\bigl(1+\E\|u_{\var}^n(t)\|_V^2\bigr).
\]
Thus the uniform Gronwall lemma, \eqref{appendix:V2-average}, and the
zero initial condition imply
\begin{equation}\label{appendix:V2-entire}
\sup_{0<\var\le1}\sup_{n\in\mathbb N}\sup_{t\ge-n}
\E\|u_{\var}^n(t)\|_V^2\le C.
\end{equation}

By It\^o's formula again, 
with the help of \eqref{eq0731:02-1}, \eqref{appendix:H2-entire},
{\bf{(H2)}} (ii) and \eqref{appendix:V2-entire},
we obtain for any small $\eta>0$,
\begin{align*}
\frac{d}{dt}\E\|u_\var^n(t)\|_V^4
&
\le-4\E\bigl(\|u_\var^n(t)\|_V^2\|Au_\var^n(t)\|^2\bigr)+4\E\left[\|u_\var^n(t)\|_V^2
\big|\langle f(t/\var,u_\var^n(t),\Law{u_\var^n(t)}),-Au_\var^n(t)\rangle\big|\right]\\
&\quad+6\E\left[\|u_\var^n(t)\|_V^2
\|g(t/\var,u_\var^n(t),\Law{u_\var^n(t)})\|_{L_2(U,V)}^2\right]\\
&
\leq -4\E\bigl(\|u_\var^n(t)\|_V^2\|Au_\var^n(t)\|^2\bigr)
+\eta\E\left(\|u_\var^n(t)\|_V^2\|Au_\var^n(t)\|^2\right)
+C_\eta(1+\E\|u_\var^n(t)\|_V^4)
\end{align*}
For every \(\delta>0\), the spectral theorem for \(-A\) gives
\[
\|u_\var\|_V^4
\le\delta\|u_\var\|_V^2\|Au_\var\|^2
+C_\delta\|u_\var\|^4.
\]
Choosing first \(\eta\) and then \(\delta\) sufficiently small, and
using \eqref{appendix:H4-entire}, we find constants \(c,C>0\) such
that
\begin{align*}
\frac{d}{dt}\E\|u_{\var}^n(t)\|_V^4
&\le-c\E\bigl(\|u_{\var}^n(t)\|_V^2
\|Au_{\var}^n(t)\|^2\bigr)+C\\
&\le-c(\lambda_1+\lambda)
\E\|u_{\var}^n(t)\|_V^4+C.
\end{align*}
Another comparison argument therefore gives
\begin{equation}\label{appendix:pullback-V4}
\sup_{0<\var\le1}\sup_{n\in\mathbb N}\sup_{t\ge-n}
\E\|u_{\var}^n(t)\|_V^4\le C.
\end{equation}

Fix \(0<\var\le1\) and \(t\in\mathbb R\). By
\eqref{appendix:pullback-V4}, the sequence
\(u_{\var}^n(t)\) is bounded in the reflexive space
\(L^4(\Omega;V)\). Choose a subsequence along which the lower limit of
the \(L^4(\Omega;V)\)-norms is attained; reflexivity gives a further
subsequence converging weakly in \(L^4(\Omega;V)\). On the other hand,
Lemma~\ref{lem:bounded-mild-solution} gives
\[
u_{\var}^n(t)\longrightarrow u_\var(t)
\quad\text{strongly in }L^2(\Omega;H).
\]
Since \(L^4(\Omega;V)\) is continuously embedded in
\(L^2(\Omega;H)\), the weak limit of the subsequence must be
\(u_\var(t)\). The weak lower semicontinuity of the
\(L^4(\Omega;V)\)-norm and \eqref{appendix:pullback-V4} now give
\[
\E\|u_\var(t)\|_V^4
\le\liminf_{n\to\infty}\E\|u_{\var}^n(t)\|_V^4\le C.
\]
The constant is independent of \(t\) and \(\var\), which proves the
first estimate. By Remark~\ref{rem0618}, the averaged coefficients
satisfy the same assumptions; applying the identical pullback
argument to \(\bar u(t;-n,0)\) proves the second estimate.
\end{proof}

\begin{lem}\label{prop:neural-field-example}
Consider \eqref{eq:app-model}.
Under \(\eqref{example-condition}\), assumptions \ref{item:H-f}–\ref{item:H-d} hold, and the constants therein can be chosen so that
\[
\lambda_1+\lambda>2L_f+2L_g^2,\qquad
2(\lambda_1+\lambda)>\max\{c_1+c_2,c_1+2L_g^2\}.
\]
\end{lem}
\begin{proof}
Recall that
\[
f(t,x,\mu):=a(t)Bx+b(t)h+m(t)\int_H z\mu(dz),
\qquad
g(t,x,\mu):=\left(1+\frac{1}{1+|t|^\iota}\right)Q.
\]
Set $q(t):=1+(1+|t|^\iota)^{-1}$. Since the almost periodic
functions \(\gamma,a,b,m\) have uniform mean values and \(q-1\)
vanishes at infinity, the averaging verification below is uniform with
respect to the initial time.

Since $b$ is bounded and $h\in H$,
\[\|f(t,0,\delta_0)\|=\|b(t)h\| \le \|b\|_\infty \|h\|.\]
Moreover, for any $x,y\in H$ and $\mu_1,\mu_2\in\mathcal P_2(H)$, we have
\begin{align*}
\|f(t,x,\mu_1)-f(t,y,\mu_2)\|&\le |a(t)|\,\|B(x-y)\|+|m(t)|\,\|\mathfrak m(\mu_1)-\mathfrak m(\mu_2)\| \\
&\le \|a\|_\infty\|B\|_{\cal L(H)}\,\|x-y\|+\|m\|_\infty W_2(\mu_1,\mu_2),
\end{align*}
because $\|\mathfrak m(\mu_1)-\mathfrak m(\mu_2)\|\le W_2(\mu_1,\mu_2)$,
where $\mathfrak m(\mu_1):=\int_{H} z\, \mu_1(dz)$.
Therefore, \ref{item:H-f} holds with $K_f=\|b\|_\infty\|h\|$ and
\[
L_f=\max\{\|a\|_\infty\|B\|_{\cal L(H)},\|m\|_\infty\}.
\]

Assumption \ref{item:H-g} holds with $K_g=2\|Q\|_{L_2(U,H)}$, 
$K'_g=2\|Q\|_{L_2(U,V)}$ and $L_g=L'_g=0$.
The uniform mean-value property of almost periodic functions gives
\ref{item:H-a} for $\gamma$ and $f$. Moreover, $\bar g=Q$, and
\[
\sup_{a\in\R}\frac1T\int_a^{a+T}
\|g(s,x,\mu)-Q\|_{L_2(U,H)}^2\,ds
\le \frac{\|Q\|_{L_2(U,H)}^2}{T}
\int_{-T}^{T}\frac{ds}{(1+|s|^\iota)^2}.
\]
The right-hand side tends to zero as $T\to\infty$, so \ref{item:H-a}
also holds for $g$.

Using Young's inequality, we have
for all \(u\in H\) and \(\mu\in\cal P_2(H)\), 
\begin{align*}
&
2\inpro{f(t,u,\mu),u}+\|g(t,u,\mu)\|_{L_2(U,H)}^2\\
&\le
2\|a\|_\infty\|B\|_{\cal L(H)}\,\|u\|^2
+2\|b\|_\infty\|h\|\,\|u\|  +2\|m\|_\infty\|\mathfrak m(\mu)\|\,\|u\|
+\|q\|_\infty^2\|Q\|_{L_2(U,H)}^2\\
&
\le c_0+c_1\|u\|^2+c_2\mu(\|\cdot\|^2),
\end{align*}
where
\[c_0:=\|b\|_\infty^2\|h\|^2+\|q\|_\infty^2\|Q\|_{L_2(U,H)}^2,\quad
c_1:=2\|a\|_\infty\|B\|_{\cal L(H)}+1+\|m\|_\infty, \quad 
c_2:=\|m\|_\infty.\]
Thus \ref{item:H-d} holds. Since $L_g=0$, the second bound in
\eqref{example-condition} implies both dissipativity requirements in
Theorem~\ref{thm:attractor-convergence}, while its first bound gives
$\lambda_1+\lambda>2L_f$.
\end{proof}

We finish by recalling the following lemma.
\begin{lem}\label{lem:vgronwall}
Let $T>0$, $0<\beta\leq 1$, and let $y$ be a nonnegative locally bounded function on $[0,T]$. Suppose that there exist constants $A\geq0$ and $B\geq0$ such that
\[
y(t)
\leq
A
+
B\int_0^t (t-s)^{\beta-1}y(s)\,ds,
\qquad 0\leq t\leq T.
\]
Then
\[
y(t)
\leq
A E_{\beta}\left(B\Gamma(\beta)t^\beta\right),
\qquad 0\leq t\leq T,
\]
where
\[
E_\beta(z):=
\sum_{n=0}^{\infty}\frac{z^n}{\Gamma(n\beta+1)}
\]
is the Mittag--Leffler function. In particular, there exists a constant
$C_{T,\beta,B}>0$ such that
\[
\sup_{0\leq t\leq T}y(t)
\leq
C_{T,\beta,B}A.
\]
\end{lem}

\section*{Acknowledgements}
This work is supported by National Key R\&D Program of China (No. 2023YFA1009200), 
NSFC (Grants 12531009, 12301223, 124B2009 and 11925102), 
and Liaoning Revitalization Talents Program (Grant XLYC2202042).

\end{document}